\documentclass[10pt]{amsart}
\usepackage{a4wide}
\usepackage{times}
\usepackage{amsmath,amssymb,url,stmaryrd,enumerate,bbm,enumerate}
\usepackage[francais]{babel}
\usepackage[latin1]{inputenc}
\usepackage[all]{xy}

\usepackage{xr}

\usepackage{todonotes}

\usepackage{hyperref}

\usepackage{mathrsfs}

\usepackage{yhmath}

\def\Fq{\mathbb{F}_q}
\def\Fqb{\overline{\mathbb{F}}_q}

\def\L{\mathscr{L}}
\def\La{\Lambda}
\def\A{\text{{\bf A}}}   

\def\O{\mathcal{O}} 
\def\F{\mathscr{F}}

\def\*{^\times }

\def\G{\mathscr{G}}
 
\def\Gm{\mathbb{G}_m}         
\def\l{\lambda}

\def\BB{\boldsymbol{B}}

\def\a{\alpha}

\def\s{\sigma}
\def\ph{\varphi}

\def\ssi{\Leftrightarrow}

\def\impl{\Rightarrow}

\def\drt{\rightarrow}
\def\ldrt{\longrightarrow}
\def\Q{\mathbb{Q}}

\def\Ql{\mathbb{Q}_\ell}
\def\Qlb{\overline{\mathbb{Q}}_\ell}
\def\Qp{\mathbb{Q}_p}

\def\Z{\mathbb{Z}}

\def\Gal{\text{Gal}}
\def\={\! = \!}

\def\spec{\text{Spec}}
\def\spf{\text{Spf}}

\def\E{\mathscr{E}}

\def\limp{\underset{\longleftarrow}{\text{ lim }}\;}
\def\limi{\underset{\longrightarrow}{\text{ lim }}\;}

\def\iso{\xrightarrow{\;\sim\;}}

\def\End{\text{End}}

\def\GL{\hbox{GL}}

\def\xrig{\xrightarrow}

\def\M{\mathcal{M}}

\def\X{\mathfrak{X}}
\def\GG{\Gamma}

\def\bc{\backslash}

\def\spa{\text{Spa}}

\def\et{\rm{\acute{e}t}}

\def\Div{\mathrm{Div}}

\def\DD{\mathbb{D}}

\def\LT{\mathscr{LT}}

\def\<<{\langle\langle}
\def\>>{\rangle\rangle}

\def\llparent{( \! ( }
\def\rrparent{) \! ) }
\def\Perf{\text{Perf}}
\def\Eb{\overline{E}}

\def\Pic{\mathscr{P}ic}
\def\BBB{\mathbb{B}}

\def\Ebc{\widehat{\vphantom{\rule{2pt}{7.2pt}} \smash{\overline{E}}}}

\entrymodifiers={+!!<0pt,\fontdimen22\textfont2>}

\newcommand{\mysetminusD}{\hbox{\tikz{\draw[line width=0.6pt,line cap=round] (3pt,0) -- (0,6pt);}}}
\newcommand{\mysetminusT}{\mysetminusD}
\newcommand{\mysetminusS}{\hbox{\tikz{\draw[line width=0.45pt,line cap=round] (2pt,0) -- (0,4pt);}}}
\newcommand{\mysetminusSS}{\hbox{\tikz{\draw[line width=0.4pt,line cap=round] (1.5pt,0) -- (0,3pt);}}}

\newcommand{\mysetminus}{\mathbin{\mathchoice{\mysetminusD}{\mysetminusT}{\mysetminusS}{\mysetminusSS}}}

\renewcommand{\setminus}{\mysetminus}

 \DeclareMathSymbol{B}{\mathalpha}{operators}{`B}
\author{Laurent Fargues}
\address{Laurent Fargues, CNRS, Institut de Math\'ematiques de Jussieu, 4 place Jussieu 75252 Paris}
\email{laurent.fargues@imj-prg.fr}
\thanks{L'auteur a bénéficié du support du projet ANR-14-CE25 "PerCoLaTor"}

\begin{document}

\title{Simple connexité des fibres d'une application d'Abel-Jacobi et corps de classe local}

\date{\today}

\maketitle

\renewcommand\labelitemi{\textbullet}

\newtheorem{theo}{Théorème}[section]
\newtheorem*{theon}{Théorème}

 \newtheorem{prop}[theo]{Proposition}

\newtheorem{coro}[theo]{{Corollaire}}

\newtheorem{lemme}[theo]{Lemme}
\newtheorem{question}[theo]{Question}

\newtheorem{defi}[theo]{Définition}
\newtheorem{exem}[theo]{Exemple}
\newtheorem{conj}[theo]{Conjecture}
\newtheorem{rema}[theo]{Remarque}

\markright{SIMPLE CONNEXIT\'E DES FIBRES D'UNE APPLICATION D'ABEL JACOBI ET CORPS DE CLASSE}

\begin{abstract}
On donne une démonstration du type Langlands géométrique de la conjecture de géométrisation de la correspondance de Langlands locale de l'auteur pour $\GL_1$. Pour cela on étudie en détails un certain morphisme d'Abel-Jacobi dont on montre que c'est une fibration pro-étale localement triviale en diamants simplement connexes en grand degré . Ces diamants sont des espaces de Banach-Colmez absolus épointés que l'on étudie en détails.
\end{abstract}

\setcounter{tocdepth}{1}
\tableofcontents

\section*{Introduction}

Cet article concerne le cas abélien de la conjecture de type Langlands géométrique pour la correspondance de Langlands locale  formulée par l'auteur (\cite{Geometrisation}, \cite{GeometrizationReview}). Rappelons que cette conjecture affirme que si $E$ est un corps local, $G$ un groupe réductif sur $E$ et $$\ph:W_E\drt \,^L G $$  un paramètre de Langlands discret on devrait pouvoir construire un faisceau pervers $\F_\ph$ sur le champ
$$
\text{Bun}_{G,\Fqb}
$$
des $G$-fibrés sur la courbe que l'on a introduite dans notre travail en commun avec Jean-Marc Fontaine (\cite{Courbe}). Ce faisceau pervers devrait satisfaire de nombreuses propriétés et en particulier construire les L-paquets locaux munis de leur structure interne (i.e. une paramétrisation de chaque élément du L-paquet) associés à $\ph$ pour toute forme intérieure étendue pure de $G$. Le champ $\text{Bun}_G$ lui est un champ perfectoïde pour la topologie pro-étale de Scholze. C'est un objet qui vit dans le monde des diamants introduits par Scholze (\cite{ScholzeBerkeley}, \cite{ScholzeCohomologyDiamonds}).
\\

Lorsque $G=\GL_1$ on peut déduire la conjecture de la théorie du corps de classe local. Néanmoins on cherche une preuve de ce résultat indépendante du type corps de classe géométrique. Si $X$ est une courbe propre et lisse le point principal dans la construction du faisceau automorphe associé à un système local sur $X$ est le fait que pour $d\gg 0$ le morphisme d'Abel-Jacobi
\begin{eqnarray*}
\Div^d_X & \ldrt & \text{Pic}^d_X \\
D & \longmapsto & \O(D)
\end{eqnarray*}
est une fibration localement triviale en variétés algébriques simplement connexes (des espaces projectifs).
\\
 Dans cet article on démontre un résultat analogue dans le cadre de notre conjecture et on en déduit la conjecture pour $\GL_1$ indépendamment de la théorie du corps de classe local. Cela redémontre en particulier celle-ci sous la forme du théorème de Kronecker-Weber local.
\\

Voici une description plus détaillée des principaux résultats. On note $\Fq$ le corps résiduel de notre corps local $E$. 
On définit dans la section \ref{sec: diamant de Hilbert} l'espace de modules des diviseurs effectifs de degré $d>0$ sur la courbe
$$
\Div^d
$$ 
dont on démontre que c'est un diamant. Plus précisément (prop. \ref{prop: identidification entre Div1 et Spa divise})
$$
\Div^1=\spa (E)^\diamond/\ph^\Z
$$
ce qui traduit le fait que les débasculements d'un $\Fq$-espace perfectoïde $S$ sur $E$ donnent des diviseurs de Cartier  de degré $1$ sur la courbe $X_S$. 
%On notera ici la dissymétrie par rapport au cas classique au sens où pour une courbe usuelle $X$, $\Div^1_X=X$ alors que dans notre cas 
%\begin{eqnarray*}
%X_S^{\diamond} &=& (S\times \spa (E)^\diamond )/\ph_S^\Z \\
%\Div^1_S &=& (S\times \spa (E)^\diamond )/\ph_{E^\diamond}^\Z 
%\end{eqnarray*}
%qui ne sont pas égaux, $X_S$ vivant au dessus de $\spa (E)$ tandis que $\Div^1_S$ est au dessus de $S$. Néanmoins l'espace topologique et le site étale de ces diamants non-isomorphes coïncident.
 On montre de plus (prop. \ref{prop: Div d comme quotient par le groupe symetrique}) que pour $d>0$ 
$$
\Div^d = (\Div^1)^d/\mathfrak{S}_d
$$
comme quotient pro-étale. Il s'agit d'une nouvelle application de notre résultat de factorisation des éléments primitifs obtenu avec Fontaine qui est la clef de voûte de \cite{Courbe}. Le diamant $\Div^d$ 
admet une description alternative comme espace projectif sur un espace de Banach-Colmez absolu $\BB^{\ph=\pi^d}$,
$$
\Div^d=\BB^{\ph=\pi^d}\setminus \{0\}/\underline{E}^\times
$$
via lequel le morphisme d'Abel-Jacobi en degré $d$ est  donné par 
$$
\text{AJ}^d: \BB^{\ph=\pi^d}\setminus \{0\}/\underline{E}^\times \ldrt [ \bullet /\underline{E}^\times ].
$$
Ces espaces $\BB^{\ph=\pi^d}$ ne sont pas des espaces de Banach-Colmez usuels, ce ne sont pas des diamants mais des diamants absolus, un nouveau concept dont nous avons besoin dans ce texte (cf. sec. \ref{sec: diamants absolus}). Il est  remarquable qu'une fois épointés ces diamants absolus deviennent des diamants (prop. \ref{prop: BC epointe est un diamant}). On constate donc que $\text{AJ}^d$ est une fibration pro-étale localement triviale de fibre $\BB^{\ph=\pi^d}\setminus \{0\}$. Le résultat principal de ce texte, qui est nettement plus fort que ce dont nous avons besoin afin de développer le programme de Langlands géométrique pour $\GL_1$, 
 est alors le suivant (théo. \ref{theo: theoreme principal de simple connexite}).
 
 \begin{theon}
Pour $d>1$, resp. $d>2$ si $E|\Qp$, le diamant $\BB_{\Fqb}^{\ph=\pi^d}\setminus \{0\}$ est simplement connexe au sens où tout revêtement étale fini connexe est trivial.  
 \end{theon}
 
 La preuve de ce résultat est donnée dans les sections \ref{sec: le cas d egales car} et \ref{sec: le cas d inegales}.  Le cas d'égales caractéristiques est particulièrement plus simple puisqu'alors le diamant $\BB^{\ph=\pi^d}\setminus \{0\}$ est un espace perfectoïde. Cependant l'analyse de sa preuve est particulièrement éclairante puisque celle-ci repose au final sur le théorème de pureté de Zariski-Nagata. 
La preuve que nous donnons dans le cas d'égales caractéristiques consiste essentiellement à démontrer un tel résultat de pureté pour l'inclusion
$$
\BB^{\ph=\pi^d}\setminus \{0\}\hookrightarrow \BB^{\ph=\pi^d}
$$
et à dévisser ce résultat en un résultat de pureté en géométrie rigide \og usuelle\fg{} (théo. \ref{theo: purete en geo rigide}).
\\

Une fois ce résultat établi le corps de classe géométrique se décrit de la façon suivante. On constate (sec. \ref{sec:sys locaux sur Div1}) que 
$$
\og \pi_1 (\Div^1_{\Fqb}) = W_E \fg{}
$$
du moins du point de vue dual des $\Qlb$-systèmes locaux 
(les guillemets sont là pour signifier que cela n'a pas vraiment de sens puisqu'on ne dispose pas à priori d'un bonne théorie du $\pi_1$ dans ce contexte)
i.e. les L-paramètres $\ell$-adiques correspondent aux systèmes locaux sur $\Div^1_{\Fqb}$. Partant d'un tel paramètre abélien $\ph:W_E\drt \Qlb^\times$, si $\E_\ph^{(1)}$ désigne le $\Qlb$-système local sur $\Div^1_{\Fqb}$ associé, on construit son symétrisé pour $d>0$
$$
\E_\ph^{(d)}= \big ( \Sigma^{d}_* \E_\ph^{(1)\boxtimes d} \big )^{\mathfrak{S}_d}
$$
dans la section \ref{sec:construction sym} où $\Sigma^d:(\Div^1)^d\drt \Div^d$ est le morphisme somme de $d$-diviseurs de degré $1$. Par application du théorème précédent celui-ci 
descend le long de $\text{AJ}^d$ en $\F_{\ph}^{(d)}$ sur $\Pic^{d}$. Il s'agit du faisceau $\F_{\ph |\Pic^d}$ de notre conjecture. Il s'étend alors automatiquement en tous les degrés en utilisant les structures monoïdales de $\Div$ et $\Pic$. Cela prouve que $\F_{\ph}^{(1)}$ descend le long de $\text{AJ}^1$ et donc que $\ph$ se factorise via l'application de réciprocité d'Artin (prop. \ref{prop: reciprocite Artin}).
\\

{\it Remerciements: J'aimerais remercier Peter Scholze et Arthur-César Le Bras pour des discussions sur le sujet. Je remercie également Werner Lütkebohmert de m'avoir expliqué ses résultats d'extension de faisceaux cohérents en géométrie rigide et Ofer Gabber pour des discussions concernant la section \ref{sec: purete en geo rigide}.}

\section{Rappels sur le cas \og classique\fg{}}

Soit $X$ une courbe propre et lisse sur un corps $k$.  Pour $d\geq 1$ on note 
$$
\Div^d = X^d/\mathfrak{S}_d
$$
le schéma de Hilbert des diviseurs de Cartier effectifs de degré $d$ sur $X$. 
On note également 
$$
\Pic = \coprod_{d\in \Z} \Pic^d 
$$
le champ de Picard de $X$ 
et 
$$
\text{Pic}=\coprod_{d\in \Z} \text{Pic}^d 
$$
le schéma de Picard de $X$ qui est l'espace de modules grossier de $\Pic$, $\text{Pic}^0$ étant la Jacobienne de $X$. Le morphisme 
$$
\Pic \ldrt \text{Pic} 
$$
est une $\Gm$-gerbe qui est scindée si $X$ possède un point $k$-rationnel, auquel cas $\Pic \simeq \big [\text{Pic}/\Gm\big ]$.
\\

Soit maintenant $\E$ un $\Qlb$-système local étale de rang $1$ sur $X$. Dans ce cadre là le programme de Langlands géométrique s'attache à construire un  $\Qlb$-système local $\F$ de rang $1$ sur $\text{Pic}$ satisfaisant certaines propriétés. Voici les étapes de sa construction  (\cite{LaumonEisenstein} sec. 2).
\\

La première étape  consiste à former le $\Qlb$-faisceau sur $\Div^d$, $d>0$,
$$
\E^{(d)}=\big (\Sigma^d_* \E^{\boxtimes d}\big )^{\mathfrak{S}_d}
$$
où $\Sigma^d:(\Div^1)^d\drt \Div^d$ est le morphisme somme de $d$ diviseurs de degré $1$. 
C'est un $\Qlb$-système local de rang $1$ (en toutes généralité, si $\E$ n'est plus de rang $1$, il s'agit d'un faisceau pervers) qui correspond à $\E$ via l'égalité $\pi_1 (X^d/\mathfrak{S}_d)=\pi_1(X)^{ab}$ lorsque $d>1$.
\\

La seconde étape est la suivante. 
Pour $d>0$ il y a un morphisme d'Abel-Jacobi 
\begin{eqnarray*}
AJ^d:  \Div^d &\ldrt &  \text{Pic}^d\\
D & \longmapsto & [\O(D)].
\end{eqnarray*}
Supposons $d>2g-2$, $g$ étant le genre de $X$. Si $S$ est un $k$-schéma de type fini et $\L$ un fibré vectoriel sur $X\times S$ fibre à fibre sur $S$ de degré $d$ alors d'après le théorème de Riemann Roch $R^1 f_*\L=0$ où $f:X\times S\drt S$. Le complexe parfait $Rf_*\L$ est donc un fibré vectoriel égal à $f_*\L$. 
Cette construction définit un fibré vectoriel $\mathcal{M}$ sur $\Pic^d$ et $\Div^d$ s'identifie à l'espace projectif sur $\M$. 
Plus précisément, $\End (\M)$ descend le long de $\Pic^d\drt \text{Pic}^d$ en une algèbre d'Azumaya $\mathcal{A}$ sur $\text{Pic}^d$ de classe dans le groupe de Brauer de $\text{Pic}^d$ la classe de la $\Gm$-gerbe $\Pic^d\drt \text{Pic}^d$. Dès lors
$$
\text{AJ}^d:\Div^d = \text{SB} ( \mathcal{A})\drt \text{Pic}^d
$$
la variété de Severi-Brauer associée à $\mathcal{A}$. La simple connexité des espaces projectifs implique alors que $\E^{(d)}$ descend le long de $\text{AJ}^d$ (\cite{SGA1} exp. X théo. 1.3) en un unique système local de rang $1$ 
$
\F^{(d)}
$
sur $\text{Pic}^{d}$.

\begin{rema}
Puisque $\Pic^d\drt \text{Pic}^d$ est une $\Gm$-gerbe et que $\Gm$ est connexe on a $\pi_1 (\Pic^d)= \pi_1 ( \text{Pic}^d)$. Il s'en suit que l'on peut remplacer $\text{Pic}$ par $\Pic$ partout. C'est d'ailleurs ce que l'on va faire plus tard car dans notre cas l'espace de modules grossier sera trivial et l'on aura affaire à une $\underline{E}^\times$-gerbe avec $E^\times$ totalement discontinu (cf. sec. \ref{sec: diamant de Hilbert}).
\end{rema}

La dernière étape de la construction de $\F$ consiste à utiliser la structure de groupe sur $\Pic$ afin d'étendre le système local $\coprod_{d>2g-2} \F^{(d)}$ sur $\coprod_{d>2g-2} \text{Pic}^d$ à tous les degrés. Pour cela on constate que le système local précédent est compatible à cette loi de monoïde sur $\coprod_{d>2g-2} \text{Pic}^d$. Il s'étend alors naturellement à tout $\text{Pic}$ en imposant que l'extension soit encore compatible à cette loi de groupe. 
\\

Traduite du point de vue dual des groupes fondamentaux la construction précédente se résume ainsi. On veut montrer que 
$$
\pi_1 (\text{AJ}^1):\pi_1 (X)^{ab} \iso \pi_1 (\text{Pic}^1).
$$
Pour cela on construit un isomorphisme 
$$
\pi_1 (X)^{ab} = \pi_1 ( X^d/\mathfrak{S}_d) 
$$
pour $d>1$ (\cite{SGA1} chap. IV Rem. 5.8) qui est le dual de l'opération $\E\mapsto \E^{(d)}$ précédente. 
 On montre alors que pour $d\gg 0$
$$
\pi_1 (\text{AJ}^d) : \pi_1 (X^d/\mathfrak{S}_d) \iso \pi_1 (\text{Pic}^d)=\pi_1 (\text{Pic}^1)
$$
en montrant que $\text{AJ}^1$ est une fibration étale localement triviale de fibre simplement connexe couplé à la suite exacte de Serre.
\\

Dans la suite on va mener à bien ces constructions classiques dans le cadre de notre conjecture.

\section{Le diamant de Hilbert $\Div^d$}
\label{sec: diamant de Hilbert}

\subsection{Quelques définitions}

Soit $E$ un corps local de corps résiduel le corps fini $\Fq$ et $\pi$ une uniformisante de $E$, $\O_E/\pi=\Fq$. On a donc soit $[E:\Qp]<+\infty$, soit $E=\Fq\llparent \pi\rrparent$. On note $\Perf_{\Fq}$ la catégorie des $\Fq$-espaces perfectoïdes munie de la topologie pro-étale.
\\

Rappelons que pour $S\in \Perf_{\Fq}$ on note 
$$
X_S= Y_S/\ph^\Z
$$
comme $E$-espace adique pré-perfectoïde. 
Il s'agit de la \og famille de courbes paramétrées par $S$\fg{}.
Lorsque $S=\spa (R,R^+)$ est affinoïde perfectoïde
$$
Y_S = \spa (\A_{R,R^+},\A_{R,R^+})\setminus V ( \pi [\varpi])
$$
où $\varpi \in R^{00}\cap R^\times$ est une pseudo-uniformisante,
$$
\A_{R,R^+} = \begin{cases}
 W_{\O_E} (R^+) \text{ si } E|\Qp \\
 R^+ \llbracket \pi\rrbracket \text{ si } E |\Fq \text{ auquel cas on pose } [a]=a \text{ pour } a\in R^0
\end{cases}
$$
est l'anneau $A_{\rm inf}$ de Fontaine associé à $R$ lorsque $E=\Qp$  
et $R^+=R^0$. 
Ici le Frobenius $\ph$ est celui induit par le Frobenius des vecteurs de Witt ramifiés. Via la formule 
$$
Y_S^\diamond = S\times \spa (E)^\diamond 
$$
on a $\ph= \ph_S \times Id$ comme Frobenius partiel sur ce produit.
\\

Il y a une application continue surjective
$$
\tau: |X_S|\ldrt |S|
$$
définie via les égalités 
$$
|X_S|=|X_S^{\diamond}|= 
\big |(S\times \spa (E)^\diamond )/\ph_S^\Z\big | = \big | (S\times \spa (E)^\diamond )/\ph_{E^\diamond}^\Z \big | \ldrt |S|
$$
puisque $\ph_S\circ \ph_{E^\diamond}$ est le Frobenius absolu de $S\times \spa (E)^\diamond$ qui agit trivialement sur $|S\times \spa (E)^\diamond|$. Pour tout ouvert $U$ de $S$  $$X_U=\tau^{-1}(U)$$ et si $s\in S$
$$
|X_{K(s),K(s)^+}| = \bigcap_{U\ni s} |X_U|\hookrightarrow |X_S|
$$
 où l'on note $K(s)$ le corps résiduel complété de $S$ en $s$ et $U$ parcourt les voisinages ouverts de $s$. 
% Cela signifie que $\dpt{|X_{K(s),K(s)^+}|=\bigcap_{U\ni s} |X_U|}$ et via l'inclusion $j:|X_{K(s),K(s)^+}|\hookrightarrow |X_S|$ $$\O_{X_{K(s),K(s)^+}}^+/\pi = j^* \O_{X_S}^+/\pi.$$ 
 On a ainsi 
$$
|X_S| = \bigcup_{s\in S} |X_{K(s),K(s)^+}|.
$$
Cela donne un support au fait que l'on pense à $X_S$ comme étant la famille 
$(X_{K(s),K(s)^+})_{s\in S}$. 
On utilisera le lemme suivant qui  signifie intuitivement que \og $X_S/S$ est propre\fg.

\begin{lemme}\label{lemme:tau est fermee}
L'application continue $\tau: |X_S|\ldrt |S|$ est fermée. 
\end{lemme}
\begin{proof}
Il s'agit d'une propriété locale sur $S$ que l'on peut supposer quasi-compact quasi-séparé.
Le morphisme de diamants $(S\times \spa (E)^\diamond)/\ph_{E^\diamond}^\Z \drt S$ est alors quasi-compact quasi-séparé. Il est de plus partiellement propre puisque $S\times \spa (E)^\diamond\drt S$ est partiellement propre car $\spa (E)^\diamond/\spa (\Fq)$ est partiellement propre. On en déduit que l'image d'un fermé par $\tau$ est un ensemble pro-constructible stable par spécialisations donc fermé.
\end{proof}

\begin{rema}
On peut également vérifier que $\tau$ est ouverte et que donc \og $X_S/S$ est propre et lisse\fg{}.
\end{rema}

On aura également besoin du lemme qui suit.

\begin{lemme}\label{lemme: injection fonctions fibre a fibre}
Pour $V$ un ouvert de $X_S$, resp. $Y_S$,
$$
\O(V)\hookrightarrow \prod_{x\in U} \O \big ( V\cap X_{K(\tau(x)),K(\tau (x))^+}\big ), \ \text{resp. } \O (V) \hookrightarrow \prod_{x\in V} \O \big ( V\cap Y_{K(\tau(x)),K(\tau (x))^+}\big ).
$$
\end{lemme}
\begin{proof}
Si $C=\Ebc$ alors $Y_S\hat{\otimes}_E C$ est perfectoïde. De plus  pour $s\in S$
$$
Y_{K(s),K(s)^+}\hat{\otimes}_E C = \underset{U\ni s}{\limp} Y_{U}\hat{\otimes}_E C 
$$
et donc le corps résiduel  complété de $x\in Y_S\otimes_E C$ coïncide avec celui de $x\in Y_{K(\tau(x)),K(\tau(x))^+}\hat{\otimes}_E C$. On conclut en utilisant que
$$
\O(V)\subset \O (V\hat{\otimes}_E C ) \subset \prod_{x\in V\hat{\otimes}_E C} K(x).
$$
\end{proof}

\begin{defi}
On note $\Pic$ le champ de Picard des fibrés sur la courbe, pour $S\in \Perf_{\Fq}$ 
$$
\Pic (S) =\text{groupoïde des fibrés en droites sur }X_S.
$$
\end{defi}

On a une décomposition en sous-champs ouverts/fermés
$$
\Pic =\coprod_{d\in \Z} \Pic^d.
$$
Il y a de plus des isomorphismes 
\begin{eqnarray*}
\Pic ^d &\iso & [\spa (\Fq)/ \underline{E}^\times ] \\
\L & \longmapsto & \text{Isom} ( \O(d),\L)
\end{eqnarray*}
où
\begin{itemize}
\item $[\spa (\Fq)/ \underline{E}^\times ]$ est le champ classifiant des $\underline{E}^\times$-torseur 
pro-étales
\item $\O(d)$ désigne le fibré en droites $\O_{X_S} (d)$ sur $X_S$ {\it associé au choix de $\pi$.}
\end{itemize}
En d'autres termes pour tout fibré en droites $\L$ sur $X_S$ fibre à fibre de degré $d$ sur $S$ il existe un recouvrement pro-étale $\widetilde{S}\drt S$ tel que $\L_{|X_{\widetilde{S}}} \simeq \O (d)$.
Il s'agit d'un cas particulier d'un résultat de Kedlaya-Liu (\cite{KedlayaLiuRelative1} théo. 8.5.12).

\begin{rema}
Le fait que l'espace de modules grossier de $\Pic^0$ soit trivial n'est rien d'autre qu'une traduction d'un des théorèmes principaux de \cite{Courbe} qui dit que si $X$ est la courbe algébrique associée à un corps perfectoïde algébriquement clos alors $B_{cris}^{\ph=Id} = \O (X\setminus\{\infty\})$ est un anneau principal. On a en effet $\text{Pic}^0 (X)= \mathcal{C}l ( B_{cris}^{\ph=Id})$.
\end{rema}

\begin{defi}\label{defi:Divd}
Pour un entier $d\geq 1$ on note $\Div^d$ le faisceau sur $\Perf_{\Fq}$ qui à $S$ associe les classes d'équivalences de couples $(\L,u)$ où 
\begin{itemize}
\item $\L$ est un fibré en droites sur $X_S$ fibre à fibre sur $S$ de degré $d$
\item $u\in H^0(X_S,\L)$ est non nul fibre à fibre sur $S$
\item $(\L,u)\sim (\L',u')$ si il existe un isomorphisme entre $\L$ et $\L'$ envoyant $u$ sur $u'$.
\end{itemize}
\end{defi}

Fibre à fibre sur $S$  signifie que pour tout $s\in S$, via $X_{K(s),K(s)^+}\drt X_S$, la condition est vérifiée par tiré en arrière sur $X_{K(s),K(s)^+}$.
 Un tel $u$ définit un morphisme
$\O_{X_S}\drt \L$ qui est en fait injectif. Cela est clair si $S=\spa (K,K^+)$ avec $(K,K^+)$ un corps affinoïde perfectoïde puisque dans ce cas là le lieu d'annulation d'une section non-nulle d'un fibré en droites est un sous-ensemble discret de $|X_{K,K^+}|$. Le résultat général se déduit alors du lemme \ref{lemme: injection fonctions fibre a fibre}. Puisque $\O_{X_S}\hookrightarrow \L$ les automorphismes de $(\L,u)$ sont triviaux et $\Div^d$ est bien un faisceau pro-étale comme annoncé dans la définition.
\\

Si on était dans un \og cadre classique\fg{} c'est à dire $X_S=X\times_k S$ avec $X$ une courbe propre et lisse sur le corps $k$ et $S$ un $k$-schéma la définition précédente coïnciderait avec celle d'un {\it diviseur de Cartier relatif effectif de degré $d$ sur $X_S/S$.}

\begin{defi}
Pour $d\geq 1$ on note 
\begin{eqnarray*}
\text{AJ}^d: \Div^d & \ldrt & \Pic^d \\
\ (\L,u) & \longmapsto  & \L 
\end{eqnarray*}
le morphisme d'Abel-Jacobi en degré $d$.
\end{defi}

On va étudier  ce morphisme en détails en décrivant de manières différentes $\Div^d$.

\subsection{Description de $\Div^d$ en termes d'espaces projectifs sur un Banach-Colmez absolu}

Soit $S\in \Perf_{\Fq}$. 
Par définition le site pro-étale de $X_S$, resp. $Y_S$, est formé des espaces perfectoïdes pro-étales au dessus de $X_S$, les recouvrements pro-étales étant définis de façon usuelle. Si $\F$ est un faisceau pro-étale sur $X_S$ alors ses sections globales, i.e. ses sections sur l'objet final du topos, sont données par 
 $$\F(X_S)= \F\big (X_S\hat{\otimes}_E \Ebc \, \big )^{\Gal (\overline{E} |E)}.$$ Il en est de même pour $Y_S$. Tout fibré vectoriel sur ces espaces définit un faisceau pro-étale  dont les sections globales sont les sections globales usuelles pour la topologie analytique.
\\

Le morphisme de topos associé à $\tau:|X_S|\drt |S|$ s'étend en un morphisme de topos
$$
\tau: (X_S)_{\text{pro-ét}}^{\widetilde{\ \ }} \ldrt \widetilde{S}_{\text{pro-ét}}.
$$
Cela résulte de ce que si $T/S$ est étale alors $X_T/X_S$ l'est également. Il en est de même pour $Y_S$.

%Le morphisme de topos associé à $\tau:|X_S|\ldrt |S|$ s'étend en un morphisme de topoï
%$$
%\tau : (X_S)_{\text{pro-ét}}^{\widetilde{\ }}\ldrt S_{\text{pro-ét}}^{\widetilde{\ }}
%$$
%où
%par définition $(X_S)_{\text{pro-ét}}^{\widetilde{\ }}$ est le topos quasi-pro-étale du diamant $X_S^\diamond$. 
%Plus précisément si $T/S$ est pro-étale alors $X_T/X_S$ est pro-étale et si $T/S$ est un recouvrement pro-étale alors $X_T/X_S$ en est également un. Du point de vue des diamants cela se voit en utilisant la formule 
%$$
%(X_S^\diamond)_{\text{pro-ét}} = \big (  (S\times \spa (E)^\diamond )/\ph_S^\Z \big )_{\text{pro-ét}} = \big ( (S\times \spa (E)^\diamond)/\ph_{E^\diamond}^\Z \big )_{\text{pro-ét}}
%$$
%où la seconde inégalité se déduit du fait que $\ph_S\circ \ph_{E^\diamond}$ est le Frobenius absolu du diamant $S\times \spa (E)^\diamond$ qui agit trivialement sur le site pro-étale de $S\times \spa (E)^\diamond$.

\begin{defi}
On note $\BB$ le faisceau pro-étale en anneaux sur $\Perf_{\Fq}$ défini par 
$$
\BB (S)= \O(Y_S).
$$
On note $\ph$ son endomorphisme de Frobenius.
\end{defi}

En d'autres termes pour tout $S$, $\BB_S$ est l'image directe de $\O_{Y_S}$ via $(Y_S)_{\text{pro-ét}}^{\widetilde{\ \ }}\drt \widetilde{S}_{\text{pro-ét}}$.
\\

Pour $d\geq 0$ 
on appelle le faisceau
$$
\BB^{\ph=\pi^d}
$$
un {\it espace de Banach-Colmez  absolu} (on renvoie à la section \ref{sec: diamants absolus} qui suit pour plus de détails sur la notion \og d'objet absolu \fg{}).  Pour tout $S\in \Perf_{\Fq}$
$$
\BB^{\ph=\pi^d}_S:= \BB^{\ph=\pi^d} \times_{\spa(\Fq)}  S
$$
est un espace de Banach-Colmez au dessus de $S$ (\cite{Colmez2}),
$$
\BB^{\ph=\pi^d}_S = \tau_* \O_{X_S}(d).
$$ 
C'est un diamant localement spatial sur $S$. Il est de plus {\it partiellement propre au dessus de $\spa (\Fq)$} au sens où pour $S=\spa (R,R^+)$ affinoïde perfectoïde 
$$
\BB (R,R^+) =\BB (R,R^0)
$$
et donc $\BB (R,R^+)^{\ph=\pi^d} =\BB (R,R^0)^{\ph=\pi^d}$.
 Néanmoins 
{\it $\BB^{\ph=\pi^d}$ n'est pas un diamant}. Pour $d=0$ on a $$\BB^{\ph=Id}=\underline{E}.$$ 

\begin{defi}
On note 
$$
\BB^{\ph=\pi^d}\setminus \{0\}
$$
le faisceau pro-étale 
$$
S\longmapsto \big \{ x\in \BB(S)^{\ph=\pi^d}\ |\ \forall s\in S, s^*x\neq 0\in \BB (K(s))^{\ph=\pi^d}\big \}.
$$
\end{defi}

\begin{lemme}
Pour tout $S$ on a 
$$
\big (\BB^{\ph=\pi^d}\setminus \{0\} \big )_S = \BB^{\ph=\pi^d}_S \setminus \iota (S)
$$
où $\iota:S\drt \BB^{\ph=\pi^d}_S$ est la section nulle qui est fermée dans $\BB_S^{\ph=\pi^d}$.
\end{lemme}
\begin{proof}
D'après le lemme \ref{lemme: injection fonctions fibre a fibre} on a $\BB(S)\hookrightarrow \prod_{s\in S} \BB (K(s))$.
Maintenant, si $f\in \O(Y_S)^{\ph=\pi^d}$, $V(f)\subset |Y_S|$ est fermé invariant sous $\ph$ et définit un fermé $Z=V(f)/\ph^\Z\subset X_S$. D'après le lemme \ref{lemme:tau est fermee} $U=|S|\setminus \tau (Z)$ est ouvert et on conclut.
\end{proof}

Le fait que la section nulle soit fermée se traduit de la façon suivante.

\begin{coro}
Le morphisme $\BB^{\ph=\pi^d}\drt \spa (\Fq)$ est séparé.
\end{coro}

On note $$\BB^{\ph=\pi^d}\setminus \{0\}/\underline{E}^\times$$ comme faisceau  pro-étale quotient. 

\begin{prop}
Il y a un isomorphisme de faisceaux pro-étales 
$$
\BB^{\ph=\pi^d}\setminus \{0\}/\underline{E}^\times \iso \Div^d
$$
via lequel le morphisme d'Abel-Jacobi est donné par
$$
AJ^d: \BB^{\ph=\pi^d}\setminus \{0\}/\underline{E}^\times\ldrt  [\spa(\Fq)/\underline{E}^\times].
$$
\end{prop}
\begin{proof}
Il y a un morphisme naturel $\underline{E}^{\times}$-invariant 
$$
\BB^{\ph=\pi^d}\setminus \{0\}\ldrt \Div^d
$$
qui envoie une section de $\O_{X_S}(d)$ non nulle fibre à fibre sur $S$ sur le diviseur associé donné par $\O_{X_S}\drt \O_{X_S}(d)$. Cela définit le morphisme 
$$
\BB^{\ph=\pi^d}\setminus \{0\}/\underline{E}^\times\ldrt \Div^d.
$$
Ce morphisme est un isomorphisme puisque pour tout $S\in \Perf_{\Fq}$, pro-étale localement sur $S$ tout fibré en droites de degré $d$ fibre à fibre sur $S$ est isomorphisme à $\O_{X_S} (d)$.
\end{proof}

On en déduit en particulier que pour tout $S\in \Perf_{\Fq}$, $\Div^d_S$ est un diamant. On verra plus tard qu'en fait $\Div^d$ est un diamant (prop. \ref{prop: Div d comme quotient par le groupe symetrique}). Mettons tout de même en exergue le corollaire suivant.

\begin{coro}
Le morphisme d'Abel-Jacobi $\text{AJ}^d$ est une fibration pro-étale localement triviale de fibre l'espace de Banach-Colmez absolu épointé $\BB^{\ph=\pi^d}\setminus \{0\}$.
\end{coro}

Notons le fait suivant.

\begin{lemme}
Le morphisme $\Div^d\drt \spa (\Fq)$ est séparé.
\end{lemme}
\begin{proof}
Soit $S\in \Perf_{\Fq}$ quasicompact quasi-séparé et $D_1,D_2\in \Div^d(S)$. Soit 
$$
T\ldrt S
$$
le $\underline{E}^\times\times \underline{E}^\times$-torseur des trivialisations de $(\O(D_1),\O(D_2))$. D'après le lemme 8.11 de \cite{ScholzeCohomologyDiamonds}
$$
T= T_0\underset{\underline{\O_E^\times}\times \underline{\O_E^\times}}{\times} {\underline{E}^\times\times \underline{E}^\times}
$$
où $T_0= T / \pi^\Z\times \pi^\Z$ est un $\underline{\O_E^\times}\times \underline{\O_E^\times}$-torseur pro-étale-fini. Le couples $(D_1,D_2)$ est alors donné par deux morphismes $\underline{E}^\times$-équivariants au dessus de $S$
$$
\xymatrix{
T\ar@<.8ex>[r]^-u  \ar@<-.8ex>[r]_-v & \big ( \BB^{\ph=\pi^d}\setminus \{0\} \big )_S
}
$$  
où $u$ est $\underline{E}^\times$-équivariant relativement à $E^\times \times \{1\}\subset E^\times \times E^\times$ et $v$ relativement à $\{1\}\times E^\times \subset E^\times \times E^\times$. Notons $u_0=u_{|T_0}$ et $v_0=v_{|T_0}$ 
$$
\text{Im} |u_0|, \text{Im} |v_0| \subset \big | ( \BB^{\ph=\pi^d}\setminus \{0\} \big )_S \big |
$$
sont des sous-ensembles quasi-compacts  dans l'espace localement spectral  quasi-séparé $ \big | ( \BB^{\ph=\pi^d}\setminus \{0\} \big )_S \big |$. On vérifie que $\pi^\Z$ agit de façon proprement discontinue sur $|(\BB^{\ph=\pi^d}\setminus \{0\})_S|$ et que donc il existe $n\geq 1$ tel que $\forall k\in \Z$,
$$
|k|\geq n\impl \text{Im} |u_0|\cap \pi^k \text{Im} |v_0|=\emptyset.
$$
Soit alors 
$$
Z=\bigcup_{|k| <n} |\text{Eq} (u_0,\pi^k v_0)| \subset |T_0|
$$
qui est fermé puisque $\BB^{\ph=\pi^d}_S$ est séparé sur $S$. Puisque $T_0\drt S$ est pro-étale-fini l'image de $Z$ dans $|S|$ est fermée. On conclut.
\end{proof}

\begin{rema} Comme me l'a fait remarquer Peter Scholze
 il faut prendre garde au fait que {\it le diamant $\Div^d$ n'est pas spatial} (\cite{ScholzeCohomologyDiamonds} déf. 9.17) car il n'est pas quasi-séparé, bien que le morphisme $\Div^d\drt \spa (\Fq)$ le soit. Typiquement la proposition 18.10 de \cite{ScholzeCohomologyDiamonds} ne s'applique pas à $\Div^d$ bien que $|\Div^d|$ soit formé d'un seul point.
 \\
 Par définition un objet $X$ d'un topos est quasi-séparé si pour tout $A,B$ quasi-compacts munis de morphismes $A\drt X$ et $B\drt X$, $A\times_X B$ est quasi-compact. Si $e$ est l'objet final du topos alors celui-ci n'est pas forcément quasi-séparé comme c'est le cas de $\spa (\Fq)$ qui est l'objet final du topos pro-étale. Dès lors il est tout à fait possible que $X\drt e$ soit quasi-séparé sans que $X$ ne le soit. 
\end{rema}

\subsection{Description de $\Div^1$ en termes de débasculements}
\label{sec:description de Div1 en termes de debasculements}

On a vu que $\Div^1=\BB^{\ph=\pi}\setminus \{0\}/\underline{E}^\times$. On va maintenant donner une description différente de $\Div^1$, le jeu étant dans la suite de passer d'une description à l'autre.
\\

Rappelons que 
$
\spa (E)^\diamond
$
désigne le faisceau des débasculements. Plus précisément, pour tout $S\in \Perf_{\Fq}$
$$
\spa (E)^\diamond (S) = \{ (S^\sharp,\iota) \}/\sim
$$
où $S^\sharp$ est un $E$-espace perfectoïde et $\iota: S\iso S^{\sharp,\flat}$. Lorsque $S=\spa (R,R^+)$ est affinoïde perfectoïde 
$$
\spa (E)^{\diamond} (R,R^+) = \{\xi\in \A_{R,R^0}\text{ primitif de degré }1\}/\A_{R,R^0}^\times.
$$
Rappelons ici que $\xi=\sum_{n\geq 0} [x_n]\pi^n$ est de degré $1$ si $x_0\in R^\times \cap R^{00}$ i.e. est une pseudo-uniformisante et $x_1\in (R^0)^\times$. 
Cela permet de voir un débasculement $S^\sharp$ de $S$ comme un diviseur 
$$
S^\sharp =V(\xi) \hookrightarrow Y_S
$$ 
qui définit lui-même une immersion 
$$
S^\sharp \hookrightarrow X_S
$$
par composition avec $Y_S \twoheadrightarrow X_S$ (il suffit de le vérifier localement sur $S$). 
Un tel élément primitif $\xi\in\A_{R,R^0}$ est sans torsion sur $\O_{Y_S}$. C'est en effet clair si $S=\spa (K,K^+)$ avec $K$ un corps perfectoïde. Le cas général résulte alors du lemme \ref{lemme: injection fonctions fibre a fibre}.
Le plongement 
$$
S^\sharp \hookrightarrow X_S
$$
définit donc un élément de $\Div^1(S)$.
\\

Cette construction définit un morphisme $\ph$-invariant de faisceau pro-étales 
\begin{equation}\label{eq:morphisme vers Div1}
\spa (E)^\diamond \ldrt \Div^1
\end{equation}
Le faisceau quotient sur $\Perf_{\Fq}$
$$
\spa (E)^\diamond /\ph^\Z.
$$
est un diamant  d'espace topologique sous-jacent un point. Pour tout $S\in \Perf_{\Fq}$
$$
(\spa (E)^\diamond /\ph^\Z)_S = (S\times \spa (E)^\diamond )/\ph_{E^\diamond}^\Z
$$
et le Frobenius partiel $\ph_{E^\diamond}^\Z$ agit de façon proprement discontinue sans points fixes sur $|S\times \spa (E)^\diamond |$. On en déduit que $\spa (E)^\diamond\drt \spa (E)^\diamond/\ph^\Z$ est une présentation de $\spa (E)^\diamond/ \ph^\Z$.
\\

Le morphisme (\ref{eq:morphisme vers Div1}) induit alors un morphisme de diamants
$$
\spa (E)^\diamond /\ph_{E^\diamond}^\Z \ldrt \Div^1
$$
dont on affirme que c'est un isomorphisme. On peut effectivement décrire explicitement ce morphisme. Soit $\mathscr{LT}$ un $\O_E$-module formel $\pi$-divisible de $\O$-hauteur $1$ sur $\O_E$. Lorsque $E|\Qp$ il s'agit d'un groupe de Lubin-Tate. Lorsque $E=\Fq\llparent \pi\rrparent$ on peut prendre $\LT= \widehat{\mathbb{G}}_a$ où l'action de $\pi\in O_E$ est donnée par $x\mapsto x^q +\pi x$, $x\in \widehat{\mathbb{G}}_a$. Soit 
$$
\widetilde{\LT}_{\Fq} =\underset{\times \pi}{\limp} \LT_{\Fq}
$$
le revêtement universel de $\LT_{\Fq}$ comme $E$-espace vectoriel formel (\cite{Courbe} chap. 4) $$\widetilde{\LT}_{\Fq}\simeq \spf ( \Fq \llbracket T^{1/p^{\infty}}\rrbracket ).$$
 Il y a alors un isomorphisme de périodes 
$$
\widetilde{\LT}_{\Fq}  \iso \BB^{\ph=\pi}.
$$
Fixons une clôture algébrique $\overline{E}$ de $E$ et soit $E_\infty |E$ le complété de l'extension
de $LT$ engendrée par les points de torsion de $\G$ dans $\overline{E}$. L'action de $\Gal (\overline{E}|E)$ sur $T_\pi (\G)$ est donnée par un caractère de Lubin-Tate 
$$
\chi_{\LT}: \Gal (E_\infty | E)\iso \O_E^{\times}.
$$
On a alors une identification compatible à l'action de $E^\times$ (\cite{ConfLaumon})
$$
\spf (\O_{E_\infty}^\flat) = \BB^{\ph=\pi}
$$
où $\pi$ agit sur $\O_{E_\infty}^\flat$ comme le Frobenius. Cela induit une identification
$$
\spa (E_\infty^\flat) \iso \BB^{\ph=\pi}\setminus \{0\}
$$
compatible à l'action de $E^\times$ et donc 
$$
\spa (E)^\diamond /\ph^\Z = \spa (E_\infty^\flat)/ \underline{E}^\times \iso \BB^{\ph=\pi}\setminus \{0\}/\underline{E}^\times.
$$
Résumons cela dans la proposition suivante.

\begin{prop}\label{prop: identidification entre Div1 et Spa divise}
Il y a un isomorphisme de diamants
\begin{equation}
\label{eq:iso vers Div1}
\spa (E)^\diamond/\ph^\Z \iso \Div^1 = \BB^{\ph=\pi}\setminus \{0\}/\underline{E}^\times.
\end{equation}
Le tiré en arrière à $\spa (E)^\diamond/\ph^\Z$ du $\underline{E}^\times$-torseur pro-étale
\begin{equation}
\label{eq:tire en arriere torseur}
\BB^{\ph=\pi}\setminus \{0\} \ldrt \BB^{\ph=\pi}\setminus \{0\} /\underline{E}^\times
\end{equation}
se décrit de la façon suivante. Considérons le $\underline{E}^\times$-torseur de Lubin-Tate 
$$
\mathbb{T}_{\LT} = \spa (E_\infty^\flat ) \underset{\underline{\O_E^\times}}{\times} {\underline{E}^\times} \ldrt \spa (E)^\diamond 
$$
obtenu par inflation de $\O_{E}^\times$ à $E^\times$. Il est muni de sa structure de Frobenius canonique
$$
can: \ph_{E^\diamond}^* \mathbb{T}_{\LT}\iso \mathbb{T}_{\LT}.
$$
Alors le $\underline{E}^\times$-torseur $\ph$-équivariant sur $\spa (E)^\diamond$ déduit de 
(\ref{eq:iso vers Div1}) et (\ref{eq:tire en arriere torseur}) est donné par 
$$
(\mathbb{T}_{\LT}, \pi \times can).
$$
\end{prop}

\begin{rema}
On a donc pour $S\in \Perf_{\Fq}$ 
$$
X_S^\diamond = (S\times \spa (E)^\diamond )/\ph_S^\Z\ldrt \spa (E)^\diamond
$$
tandis que 
$$
\Div^1_S = (S\times \spa (E)^\diamond) /\ph_{E^\diamond}^\Z\ldrt S.
$$
Puisque $\ph_{E^\diamond}\circ \ph_S$ est le Frobenius absolu de $S\times \spa (E)^\diamond$ il agit trivialement sur son espace topologique, resp. son site étale. On a donc 
\begin{eqnarray*}
|X_S| &=& | \Div^1_S | \\
(X_S)_{\et} &=& (\Div^1_S)_{\et}
\end{eqnarray*}
mais les diamants $X_S^\diamond$ et $\Div^1_S$ ne sont pas isomorphes. 
\end{rema}

En termes d'éléments primitifs de degré $1$ ce torseur s'interprète comme {\it un torseur des renormalisations de produits infinis non-convergents.}  Plus précisément, soit $S=\spa (R,R^+)$ affinoïde perfectoïde et $\xi\in \A_{R,R^0}$ primitif de degré $d\geq 1$.
Soit 
$$
S^\sharp \hookrightarrow X_S
$$
associé à $\xi$ défini par l'idéal localement libre de rang $1$, $\mathcal{I}\subset \O_{X_S}$.
 Le diviseur associé $$(\O_{X_S}\hookrightarrow \mathcal{I}^{-1})\ \in \Div^1(S)$$ 
se calcule concrètement de la façon suivante. Quitte à multiplier $\xi$ par une unité de $\A_{R,R^0}^\times$ on peut supposer que via l'application de réduction $$W_{\O_E} ( R^0)\drt W_{\O_E} ( R^0/R^{00}), \text{ resp.} R^0 \llbracket \pi\rrbracket \drt R^0/R^{00}\llbracket \pi\rrbracket\text{ si } E=\Fq\llparent \pi\rrparent,$$ on a  $\xi\equiv \pi.$
\\

Notons $\L= \O_{Y_S}$  muni de la structure de Frobenius
$$
u:\ph^*\L\hookrightarrow \L
$$
donnée par la multiplication par $\xi$. L'objet $(\L,u)$ est un Shtuka sur $Y_S$ (ce que l'on appelle un $\ph$-module dans \cite{ConfLaumon} déf. 4.28). On peut alors lui associer le système inductif/projectif comme dans la proposition 4.29 de \cite{ConfLaumon}
$$
\dots \ldrt \ph^{n+1*}\L\xrig{\ \ph^{n*} u\ } \ph^{n*}\L \xrig{\ \ph^{(n-1)* u\ }} \ph^{(n-1)*} \L \ldrt \dots  \ \ \ n\in \Z.
$$
En restriction à tout ouvert quasi-compact de $Y_S$ se système inductif/projectif est essentiellement constant et donc 
$$
\L_\infty = \underset{n\geq 0}{\limp} \ph^{n*}\L \hookrightarrow \underset{n\leq 0}{\limi} \ph^{n*}\L=\L^\infty.
$$
est un monomorphisme de fibrés en droites $\ph$-équivariants que l'on identifie à un monomorphisme de
fibrés sur $X_S$. On a alors 
$$
(\L_\infty\otimes \L^{\infty,\vee} \hookrightarrow \O_{X_S} )\simeq (\mathcal{I}\hookrightarrow \O_{X_S} ).
$$

Grâce à l'hypothèse faite sur $\xi$ le produit infini
$$
\Pi^+ (\xi) = \prod_{n\geq 0} \frac{\ph^n(\xi)}{\pi} \in \O(Y_S)
$$
est convergent dans l'algèbre de Fréchet $\O(Y_S)$. Il satisfait l'équation fonctionnelle
$$
\frac{\xi}{\pi}.
\ph (\Pi^+(\xi)) =\Pi^+(\xi)
$$
et  donc 
$$
\Pi^+ (\xi) \in H^0 ( X_S, \L_\infty (d)).
$$
Cela définit un isomorphisme 
$$
\O_{X_S} (-d)\iso \L_\infty.
$$
On a donc une identification de torseurs pro-étales sur $S$
$$
\text{Isom} ( \O(-1),\mathcal{I}) = \text{Isom} (\O, \L^{\infty}).
$$
Afin de trouver une section globale de $\L^\infty$ trivialisant celui-ci on aimerait former
$$\og
\Pi^- (\xi) = \prod_{n<0} \ph^n (\xi) \fg
$$
mais malheureusement ce produit infini n'est pas convergent. Analysons la structure de celui-ci par approximations successives. Si $\xi \equiv a\text{ mod }\pi$ avec $a\in R^{00}\cap R^\times$
alors 
\begin{equation} \label{eq:renormalisation Kummer}
\og 
\prod_{n<0} \ph^n (\xi) \equiv a^{\frac{1}{q}+\frac{1}{q^2}+\dots} \equiv a^{\frac{1}{q-1}} \text{ mod }\pi \fg
\end{equation}
qui est une solution formelle d'une equation de Kummer que l'on peut résoudre après un revêtement étale fini de $S$. Si $x\in 1+\pi^k \A_{R,R^0}$, $k\geq 1$, avec $x\equiv 1+ [b] \pi^k \text{ mod } 1+ \pi^{k+1}\A_{R,R^0}$ alors
\begin{equation}\label{eq:renormalisation Artin Schreier}
\og 
\prod_{n<0} \ph^n (x) \equiv 1+ \big [ b^{\frac{1}{q}} + b^{\frac{1}{q^2}} + \dots ] \pi^k \text{ mod } 1+  \pi^{k+1} \A_{R,R^0}\fg
\end{equation}
où $\og b^{\frac{1}{q}} + b^{\frac{1}{q^2}}+\dots \fg$ est un solution formelle d'une équation d'Artin-Schreier que l'on peut résoudre après un revêtement étale fini de $S$.
\\

En d'autres termes si $\mathbb{T}\ldrt \spa (E)^\diamond$ est notre $\underline{E}^\times$-torseur
déduit de (\ref{eq:iso vers Div1}) et (\ref{eq:tire en arriere torseur}):
\begin{itemize}
\item le $\Fq^\times$-torseur $$\mathbb{T}/\pi^\Z.(1+\pi\O_E)\drt \spa (E)^\diamond$$ est \og de type Kummer obtenu par renormalisation de produits du type (\ref{eq:renormalisation Kummer})\fg{}
\item pour $k\geq 1$ le $\Fq$-torseur  
$$
\mathbb{T}/\pi^\Z.(1+\pi^{k+1}\O_E) \ldrt \mathbb{T}/\pi^\Z.(1+\pi^k \O_E)
$$
est \og de type Artin-Schreier obtenu par renormalisation de produits du type (\ref{eq:renormalisation Artin Schreier}).\fg{}
\end{itemize}

\subsection{Description de $\Div^d$ comme puissance symétrique de $\Div^1$}

On va maintenant montrer que $\Div^d$ est la $d$-ème puissance symétrique de $\Div^1$. Cette propriété est fondamentale dans la suite de ce texte puisqu'elle est à la base de la construction de la section \ref{sec:construction sym}, le principe étant alors de jouer entre les deux descriptions de $\Div^d$, $\Div^d=\BB^{\ph=\pi^d}\setminus \{0\}/\underline{E}^\times$ et $\Div^d=d$-ième puissance symétrique de $\Div^1$.
\\

On dispose d'une structure naturelle de monoïde commutatif sur 
$$
\coprod_{d>0} \Div^d.
$$

Le résultat clef suivant est une traduction \og en famille\fg{} de notre résultat clef de factorisation 
obtenu avec Fontaine (\cite{Courbe} chap. 2 sec. 2.4).

\begin{prop}\label{prop: Div d comme quotient par le groupe symetrique}
Pour $d>0$ l'application somme de $d$ diviseurs de degré $1$, $\Sigma^d: (\Div^1)^d\drt \Div^d$ est un morphisme surjectif quasi-pro-étale de diamants qui induit un isomorphisme 
$$
(\Div^1)^d /\mathfrak{S}_d \iso \Div^d
$$
où le quotient est un quotient de faisceaux pro-étales.
\end{prop}
\begin{proof}
Soit $S\in \Perf_{\Fq}$ strictement totalement discontinu muni d'un morphisme 
$$S\drt \Div^d.$$
Montrons que dans le diagramme cartésien
$$
\xymatrix{
T \ar[d] \ar[r]^-f & S \ar[d] \\
(\Div^1)^d  \ar[r]^-{\Sigma^d} & \Div^d 
}
$$
le morphisme $T\drt S$ est -pro-étale quasicompact surjectif entre espaces perfectoïdes et donc un épimorphisme pro-étale. 
Le diagramme précédent se réécrit en un diagramme cartésien
$$
\xymatrix{
T \ar[d] \ar[r]^-f & S \ar[d] \\
(\Div^1)^d_S  \ar[r]^-{\Sigma^d} & \Div^d_S .
}
$$
L'avantage de cette réécriture est que maintenant $(\Div^1)^d_S$ et $\Div^d_S$ sont des diamants spatiaux séparés sur $S$ et donc $T$ est un diamant spatial et $f$ est séparé. On peut alors appliquer la proposition 11.6 de \cite{ScholzeCohomologyDiamonds}  pour conclure que $f$ est un morphisme affinoïde pro-étale entre espaces affinoïdes perfectoïdes. D'après le théorème 6.2.1 de \cite{Courbe} $f$ est surjectif au niveau des $(C,C^+)$-points pour tout corps affinoïde perfectoïde algébriquement clos $(C,C^+)$ 
 (nos diamants sont partiellement propres et leurs $(C,C^+)$-points coïncident avec leurs $(C,\O_C)$-points). 
 On en déduit que $\Sigma^d$ est quasi-pro-étale et est un épimorphisme de faisceaux pro-étales.

Il y a un morphisme 
$$
g:(\Div^1)^d\times \mathfrak{S}_d \ldrt (\Div^1)^d\times_{\Div^d} (\Div^1)^d.
$$ 
Appliquant la même technique que précédemment, pour tout $S$ strictement totalement discontinu et tout morphisme $S\drt (\Div^1)^d\times_{\Div^d} (\Div^1)^d$  on vérifie que le morphisme 
déduit par tiré en arrière par $g$ vers $S$ est affinoïde pro-étale surjectif. On conclut que $g$ est un épimorphisme de faisceaux pro-étales. Cela permet de conclure.
\end{proof}

%Voici une légère amélioration du lemme 5.19 de \cite{ScholzeCohomologyDiamonds}.
%
%\begin{lemme}\label{prop:morphisme quasi pro etale avec diamant a la source}
%Soit $f:X\drt Y$ où $X$ est un diamant spatial et $Y$ un espace perfectoïde strictement totalement discontinu. Supposons que pour tout corps perfectoïde algébriquement clos $C$ et tout morphisme $\spa (C,\O_C)\drt Y$ la fibre $X\times_Y \spa (C,\O_C)$ soit profinie i.e. de la forme $\underline{A}_{\spa (C,\O_C)}$ avec $A$ profini.
%Si $f$ est quasi-compact séparé alors  $X$ est  perfectoïde et $f$ est affinoïde pro-étale.
%\end{lemme}
%\begin{proof}
%On suit la preuve de 5.19 \cite{ScholzeCohomologyDiamonds} pas à pas. On a une factorisation
%$$
%\xymatrix{
%X \ar[d]^-f \ar[r]^-g & Y\times_{\underline{\pi_0(Y)}} \underline{\pi_0(X)}=:Y' \ar[ld]_-{h} \\
%Y
%}
%$$
%Le morphisme $h$ est un un morphisme affinoïde pro-étale entre espaces perfectoïdes  strictement totalement discontinus. Il este à voir que $g$ est une immersion en utilisant que $\pi_0(g)$ est bijectif et $g$ est séparé quasi-compact de fibres  un ensemble profini au dessus des points maximaux. Puisque $X$ est spatial l'image de $|g|$  est pro-constructible généralisant et définit donc un sous-espaces affinoïde perfectoïde $Y''$ de $Y'$ (\cite{ScholzeCohomologyDiamonds} lemme 5.7). Afin de vérifier que $g:X\xrig{\sim} Y''$ on peut appliquer le lemme 9.11 de \cite{ScholzeCohomologyDiamonds} en suivant la preuve du lemme 5.19.
%\end{proof}

On en déduit le résultat suivant qui n'était pas évident à priori.

\begin{prop}\label{prop: BC epointe est un diamant}
Pour $d>0$ le diamant absolu $\BB^{\ph=\pi^d}\setminus \{0\}$ est un diamant.
\end{prop}
\begin{proof}
Le morphisme $\BB^{\ph=\pi^d}\setminus \{ 0\}\drt \Div^d$ est pro-étale surjectif et on peut appliquer la proposition 9.6 de \cite{ScholzeCohomologyDiamonds}.
\end{proof}

\begin{rema}
On verra dans la section \ref{sec: le cas d egales car} que lorsque $E=\Fq\llparent \pi\rrparent$ alors $\BB^{\ph=\pi^d}\setminus \{0\}$ est même un espace perfectoïde.
\end{rema}

\section{$\Qlb$-systèmes locaux sur $\Div^1_{\Fqb}$ et représentations du groupe de Weil}
\label{sec:sys locaux sur Div1}

Soit $\ell$ un nombre premier éventuellement égal à $p$. 
Soit $D$ un diamant. 
Si $[\Q_\l:\Ql]<+\infty$, par définition $\Q_\l$-système local sur $D$ est un faisceau de $\underline{\Q_{\l}}$-modules localement constant de dimension finie sur le site quasi-pro-étale de $D$. 
 Par définition un $\Qlb$-système local sur $D$ est \og un objet de la forme $\F\otimes_{\Q_\l} \Qlb$ pour $\F$ un $\Q_{\l}$-système local et $\Q_\l\subset \Qlb$ comme précédemment.\fg{}
\\

Fixons une clôture algébrique $\Eb$ de $E$ de corps résiduel $\Fqb$ et soit $\breve{E}$ le complété de l'extension maximale non-ramifiée de $E$ dans $\Eb$ dont on note $\s$ le Frobenius. On note 
$
W_E
$
le groupe de Weil de $E$. La proposition qui suit dit en quelques sortes que 
$$
\og \pi_1 (\Div^1_{\Fqb}) =W_E \fg{}.
$$

\begin{prop}\label{prop: identification Qlb sys locaux et rep groupe de Weil}
La catégorie des $\Qlb$-systèmes locaux sur $\Div^1_{\Fqb}$ s'identifie à celle des $\Qlb$-représentations continues de $W_E$.
\end{prop}
\begin{proof}
Puisque pour tout $S\in \Perf_{\Fq}$ le morphisme $Y_S\drt X_S$ est un isomorphisme local la catégorie des $\Qlb$-systèmes locaux sur $$Div^1_{\Fqb}= (\spa (E)^\diamond \times \spa (\Fqb) )/\ph_{E^\diamond}^\Z$$ s'identifie à celle des $\Qlb$-systèmes locaux $\ph_{E^\diamond}$-equivariants sur  
$$
\spa (E)^\diamond \times \spa (\Fqb) .
$$
Puisque le Frobenius absolu $\ph_{E^\diamond}\circ \ph_{\Fqb}$ de ce diamant agit trivialement sur le site quasi-pro-étale cette dernière catégorie s'identifie à la catégories des $\Qlb$-systèmes locaux $\ph_{\Fqb}$-équivariants sur ce même diamant. Remarquons maintenant que 
$$
\spa (E)^\diamond \times \spa (\Fqb ) = \spa (\breve{E})^\diamond
$$
et que via cette identification $Id\times \ph_{\Fqb}=\s$. Notre catégorie s'identifie donc à celle des $\Qlb$-systèmes locaux $\s$-équivariants sur $\spa (\breve{E})^\diamond$. 
\end{proof}

\begin{rema}
La proposition précédente est la raison principale pour laquelle on ne travaille pas avec des coefficients de torsion. L'apparition du groupe de Weil et non du groupe de Galois dans cet énoncé est un indice sérieux 
que l'objet $\Div^1_{\Fqb}$ est le bon objet à considérer et non $\spa (E)^\diamond$ dont le $\pi_1$ serait $\Gal (\overline{E}|E)$.
\end{rema}

Soit maintenant $\chi:E^\times \drt \Ql^\times$ un caractère. Il définit un $\Qlb$-système local de rang $1$ $\E_\chi$ sur 
$$
\BB_{\Fqb}^{\ph=\pi}\setminus \{0\}/\underline{E}^\times
$$
par poussé en avant par $\chi$ du $\underline{E}^\times$-torseur $\BB_{\Fqb}^{\ph=\pi}\setminus \{0\}\drt \BB_{\Fqb}^{\ph=\pi}\setminus \{0\}/\underline{E}^\times$. La proposition qui suit dit en quelques sortes que le morphisme 
$$
\og W_E=\pi_1 (\Div^1_{\Fqb}) \xrig{\ \pi_1(\text{AJ}^1) \ } \pi_1 (\Pic^1_{\Fqb})=E^\times\fg{}
$$
est donné par l'inverse de l'application de réciprocité d'Artin.

\begin{prop}\label{prop: reciprocite Artin}
Via l'identification $\Div^1_{\Fqb}= \BB_{\Fqb}^{\ph=\pi}\setminus \{0\}/\underline{E}^\times$ 
le caractère de $W_E$ associé à $\E_\chi$ est obtenu en composant $\chi^{-1}$ avec l'inverse de l'application de réciprocité d'Artin $\text{Art}:E^\times\xrig{\sim} W_E^{ab}$.
\end{prop}
\begin{proof}
  Fixons un groupe de Lubin-Tate sur $\O_E$ associé au choix de $\pi$ et soit $\chi_{\LT}:\Gal (\Eb|E)\drt \O_E^\times$ le caractère de Lubin-Tate associé. Avec les notations de la section \ref{sec:description de Div1 en termes de debasculements} on a $$W_E^{ab}\subset \Gal (\Eb |E )^{ab} = \Gal (\breve{E}_\infty |E)
  =\Gal (\breve{E}_\infty |E_\infty)\times \Gal (E_\infty |E).$$
  ce qui permet de définir canoniquement $\s\in W_E^{ab}$. On a alors 
  \begin{eqnarray*}
  \forall \tau\in I_E,\ Art^{-1}(\tau) &=& \chi_{\LT} (\tau)^{-1} \\
  Art^{-1} (\s) &=& \pi.
  \end{eqnarray*}
 On applique maintenant la proposition \ref{prop: identidification entre Div1 et Spa divise}. 
On part du $\underline{E}^\times$-torseur $\ph_{E^\diamond}$-équivariant 
$$
\mathbb{T}_{\LT,\Fqb}:=\mathbb{T}_{\LT}\times  \spa (\Fqb) \ldrt \spa (E)^\diamond \times \spa (\Fqb)
$$  
muni de 
$$
\pi\times can:
\ph_{E^\diamond}^* \mathbb{T}_{\LT,\Fqb} \iso \mathbb{T}_{\LT,\Fqb}
$$  
On lui appliquer $\ph_{\Fqb}^*$ pour obtenir
$$
\mathbb{T}_{\LT,\Fqb}=(\ph_{\Fqb}\ph_{E^\diamond})^* \mathbb{T}_{\LT,\Fqb} \iso \ph_{\Fqb}^*\mathbb{T}_{\LT,\Fqb}
$$
dont on prend l'inverse. La donnée de descente obtenue donc sur $$\mathbb{T}_{\LT}\otimes_E \breve{E}=\mathbb{T}_{\LT, \Fqb}\drt \spa (\breve{E})^\diamond,$$ $\s^*\mathbb{T}_{\LT,\Fqb}\xrig{\sim}\mathbb{T}_{\LT,\Fqb}$, est donnée par $\pi^{-1}$ fois la donnée de descente canonique.
Le résultat s'en déduit.
\end{proof}

\section{Construction du symétrisé d'un système local abélien sur $\Div^1_{\Fqb}$}
\label{sec:construction sym}

On montre dans cette section que la construction qui à un $\Qlb$-système local de rang $1$
sur $\Div^1_{\Fqb}$ associe son symétrisé sur $\Div^d_{\Fqb}$ a bien un sens comme dans le cas \og classique\fg{} i.e. on a  un morphisme
$$
\og \pi_1 ( \Div^d_{\Fqb} )\ldrt \pi_1 (\Div^1_{\Fqb})^{ab}\fg{}
$$
(dont on montrera au final que c'est un isomorphisme). 

\subsection{Construction}

Dans cette section on utilise constamment le fait que les morphismes étales finis entre espaces perfectoïdes satisfont la descente pro-étale (\cite{ScholzeCohomologyDiamonds} prop. 7.7). Par exemple, si $D$ est un diamant et $\Gamma$ un groupe fini tout $G$-torseur étale au dessus de $D$ est représentable par un diamant étale fini au dessus de $D$.
\\

Soit $G$ un groupe abélien fini et $$U\drt \Div^1_{\Fqb}$$ un $G$-torseur  étale.
Pour $d>0$, $$U^d\drt (\Div^1_{\Fqb})^d$$ est un $G^d$-torseur que l'on peut pousser en avant via l'application somme $s:G^d\drt G$ et obtenir un $G$-torseur
\begin{equation} \label{eq:G torseur somme }
U^d\underset{G^d,s}{\times} G \ldrt (\Div^1_{\Fqb})^d.
\end{equation}
Il est muni d'une action de $\mathfrak{S}_d$ compatible à celle sur $(\Div^1_{\Fqb})^d$.

\begin{lemme}
Le quotient pro-étale par $\mathfrak{S}_d$
\begin{eqnarray}\label{eq:G torseur divise par Sd}
\big ( U^d\underset{G^d,s}{\times} G\big ) / \mathfrak{S}_d \ldrt (\Div^1_{\Fqb})^d/\mathfrak{S}_d
\end{eqnarray}
est un $G$-torseur étale qui descend le $G$-torseur (\ref{eq:G torseur somme }) via $(\Div^1_{\Fqb})^d\drt (\Div^1_{\Fqb})^d/\mathfrak{S}_d$.
\end{lemme}
\begin{proof}
Il suffit de vérifier que le diagramme 
$$
\xymatrix{
U^d\underset{G^d,s}{\times} G \ar[d] \ar[r] & \big ( U^d\underset{G^d,s}{\times} G \big )/\mathfrak{S}_d \ar[d] \\
(\Div^1_{\Fqb})^d \ar[r] & (\Div^1_{\Fqb})^d/\mathfrak{S}_d.
}
$$
est cartésien.
En effet, si c'est le cas alors (\ref{eq:G torseur divise par Sd}) est un $G$-torseur étale localement pour la topologie quasi-pro-étale au dessus du diamant $(\Div^1_{\Fqb})^d/\mathfrak{S}_d$ est est donc un $G$-torseur étale par descente des morphismes étales finis pour la topologie pro-étale.  La surjectivité du morphisme 
$$
f:U^d\underset{G^d,s}{\times} {G} \ldrt (\Div^1_{\Fqb})^d \underset{(\Div^1_{\Fqb})^d/\mathfrak{S}_d}{\times} \big ( U^d\underset{G^d,s}{\times} G \big )/\mathfrak{S}_d 
$$
est immédiate. Pour l'injectivité, remarquons que l'on a 
$$
 U^d\underset{G^d,s}{\times} G   = H\bc U^d
$$
où $H=\{(g_1,\dots,g_d)\in G^d \ | \ \prod_{i=1}^d g_i=1\}$. On note $[x_1,\dots,x_d]$ la classe d'une section de $U^d$ modulo l'action de $H$. Si $f([x_1,\dots,x_d])=f([y_1,\dots,y_d])$ il existe (étale localement) $\s\in \mathfrak{S}_d$ et $(g_1,\dots,g_d)\in G^d$ tels que 
\begin{eqnarray*}
[y_1,\dots,y_d] &=& [x_{\s (1)},\dots,x_{\s (d)}]  \\
(y_1,\dots,y_d) &=& (g_1.x_1,\dots, g_d.x_d).
\end{eqnarray*}
Il existe donc (toujours étale localement) $h_1,\dots,h_d\in G$ vérifiant $\prod_{i=1}^d h_i=1$ tels que 
$$
(h_1.y_1,\dots,h_d.y_d) = ( x_{\s(1)},\dots,x_{\s(d)}).
$$
On obtient alors pour tout indice $i$, $h_i g_i.x_{\s(i)} = x_i$. En itérant $d$-fois cette relation on en déduit que 
$$
\big (\prod_{i=1}^d h_i g_i\big ).x_1=x_1
$$
et que donc $\prod_{i=1}^d h_i g_i=1$ puisque l'action de $G$ sur $U$ est libre. Cela permet de conclut que $(g_1,\dots,g_d)\in H$ ce qui démontre l'injectivité.
\end{proof}

\begin{rema}
La preuve précédente s'applique dans un contexte très général. Par exemple, si $X$ est une courbe projective lisse sur un corps, $\Div^1=X$ et $\Div^d=X^d/\mathfrak{S}_d$ (quotient fppf) puisque les morphismes étales finis satisfont à la descente pour la topologie fppf.
\end{rema}

\begin{prop}
Soit $\La$ un anneau fini et $\E$ un faisceau de $\La$-modules localement constant libre de rang $1$ sur $(\Div^1_{\Fqb})_{\et}$. Pour $d>0$ le faisceau
$$
\big ( \Sigma^d_* \E^{\boxtimes d} \big )^{\mathfrak{S}_d}
$$
est localement constant libre de rang $1$ sur $(\Div^d_{\Fqb})_{\et}$.
\end{prop}
\begin{proof}
Soit $$U=\text{Isom} (\underline{\La},\E)\ldrt \Div^1_{\Fqb}$$ le $\La^\times$-torseur sur $\Div^1_{\Fqb}$ des trivialisations de $\E$.  
On a 
$$
( \Sigma^d_* \E^{\boxtimes d} )_{| ( U^d \underset{G^d,s}{\times } G)/\mathfrak{S}_d } = f_* \big (\E^{\boxtimes d}\big )_{| U^d\underset{G^d,s}{\times} G}.
$$
où $$f:  U^d \underset{G^d,s}{\times } G \ldrt  ( U^d \underset{G^d,s}{\times } G)/\mathfrak{S}_d .$$
 On a $(\E^{\boxtimes d})_{U^d}\simeq \underline{\La}$ et l'action de monodromie de $G^d\drt \La^\times$ déduite se factorise via le noyau de $s$. On en déduit que 
 $$
( \E^{\boxtimes d} )_{| U^d\underset{G^d,s}{\times } G} \simeq \underline{\La}.
 $$
 L'action de $\mathfrak{S}^d$ déduite sur $\underline{\La}$ est triviale et 
 $$
( f_* \underline{\La} )^{\mathfrak{S}_d} =\underline{\La}
 $$
 puisque si 
 $$
 V\ldrt (U^d \underset{G^d,s}{\times} G)/\mathfrak{S}_d
 $$ est étale 
 $$
( f_* \underline{\La} )^{\mathfrak{S}_d} = \mathscr{C} ( |V|/\mathfrak{S}_d, \La)
 $$
 et 
 $$
 |V|/\mathfrak{S}_d \iso \big | (U^d \underset{G^d,s}{\times} G)/\mathfrak{S}_d \big |.
 $$
\end{proof}

On peut étendre la construction précédente du cas des coefficients de torsion au cas des $\Qlb$-coefficients. En effet, si $|\Q_\l:\Ql]<+\infty$ alors la construction s'étend automatiquement au cas des $\Z_\l$-systèmes locaux de rang $1$. Cela provient du fait que si $\E$ est un tel $\Z_\l$-système local de rang $1$ alors 
$$
\E= \underset{n\geq 1}{\limp} \E/\ell^n
$$
et $\E/\ell^n$ est un faisceau étale localement constant de rang $1$ sur $\Z_\l/\ell^n\Z_\l$ pour la topologie étale (utiliser le fait que les morphismes étales finis satisfont la descente pour la topologie pro-étale). Maintenant si $\E$ est un $\Q_\l$-système locale de rang $1$ qui n'est pas de la forme $\E_0\otimes_{\Z_\l} \Q_\l$ on peut procéder de la façon suivante. D'après le lemme 8.11 de \cite{ScholzeCohomologyDiamonds} le $\underline{\Q_\l^\times}$-torseur pro-étale des trivialisations de $\E$ est de la forme 
$$
T= T_0\underset{\underline{\Z_\l^\times}}{\times} \underline{\Q_\l^\times} = \coprod_\Z T_0
$$
où $T_0=T/\varpi_\l^\Z$ est pro-étale-fini. Utilisant cela on vérifie la proposition suivante.

\begin{prop}\label{prop: construction du symmetrise sur Qlb}
Soit $\E$ un $\Qlb$ système local de rang $1$ sur $\Div^1_{\Fqb}$. Alors pour $d>0$ le faisceau quasi-pro-étale
$$
\E^{(d)}:= \big ( \Sigma^d_* \E^{\boxtimes d} \big )^{\mathfrak{S}_d}
$$
est un $\Qlb$ système local de rang $1$ sur $\Div^d_{\Fqb}$. 
\end{prop}

\subsection{Compatibilité au tiré en arrière via l'application d'Abel-Jacobi}

On vérifie maintenant que le diagramme \og
$$
\xymatrix{
\pi_1 (\Div^d_{\Fqb})^{ab} \ar[rr] \ar[d]_-{\pi_1(\text{AJ}^1)} & & \pi_1 (\Div^1_{\Fqb})^{ab} \ar[d]^-{\pi_1(\text{AJ}^d)} \\
\pi_1 (\Pic^1_{\Fqb}) \ar@{=}[r] & E^\times \ar@{=}[r] & \pi_1 (\Pic^d_{\Fqb}) 
} 
$$
\fg{} commute où la flêche horizontale du haut est duale à l'opération de symétrisation d'un $\Qlb$-système local de rang $1$.

\begin{prop}\label{prop: compatibilite tire en arriere}
Pour $\chi:E^\times\drt \Qlb^\times$ et $d>0$ soit $\F_\chi^{(d)}$ le système local associé sur $\mathscr{P}ic^{(d)}=[\spa (\Fqb)/\underline{E}^\times]$.  Notons $\E_{\chi}^{(1)}=\text{AJ}^{1*} \F_{\chi}^{(1)}$.
On a alors 
$$
\E_{\chi}^{(d)} :=
\big (\Sigma^d_* \F_{\chi}^{(1)\boxtimes d}\big )^{\mathfrak{S}_d} = \text{AJ}^{d*} \F_{\chi}^{(d)}.
$$ 
\end{prop}
\begin{proof}
Il y a un diagramme commutatif 
$$
\xymatrix{
(\Div^1_{\Fqb} )^d \ar[r]^-{\ \Sigma^d \ } \ar[d]_-{(\text{AJ}^1)^d} & \Div^d_{\Fqb} \ar[d]^-{\text{AJ}^d} \\
\big ( \Pic^1\big )^d \ar[r]^-{\Pi^d} & \Pic^d
}
$$
où $\Pi^d (\L_1,\dots,\L_d)=\L_1\otimes\dots\otimes \L_d$. De plus, $\Pi^d$ s'identifie au morphisme
$$
\big ( \BB_{\Fqb}^{\ph=\pi}\setminus \{0\}/\underline{E}^\times\big )^d\ldrt \big ( \BB_{\Fqb}^{\ph=\pi^d}\setminus \{0\}\big ) /\underline{E}^\times
$$
induit par la multiplication de $d$-éléments de $\BB^{\ph=pi}_{\Fqb}$. On en déduit que 
$$
\Pi^{d*} \F_\chi^{(d)} = \F_\chi^{(1)}\boxtimes \dots \boxtimes \F_{\chi}^{(1)}.
$$ 
On obtient alors 
\begin{eqnarray*}
\Sigma^{d*} \text{AJ}^{d*} \F_{\chi}^{(d)} &=& (\text{AJ}^{1})^{d*} \Pi^{d*} \F_\chi^{(d)} \\
&=&  (\text{AJ}^{1})^{d*}     \F_\chi^{(1)\boxtimes d} \\
&=& \big ( \text{AJ}^{1*} \F_\chi^{(1)} \big)^{\boxtimes d}.
\end{eqnarray*}
On peut alors appliquer le lemme qui suit \ref{lemme: caracterisation faisceaux descendus} pour conclure.
\end{proof}

\begin{lemme}\label{lemme: caracterisation faisceaux descendus}
Pour $\E$ un $\Qlb$-système local de rang $1$ sur $Div^d_{\Fqb}$ on a 
$$
\E= \big ( \Sigma^d_* \Sigma^{d*} \E\big )^{\mathfrak{S}_d}.
$$
\end{lemme}
\begin{proof}
Il y a un morphisme naturel obtenu par adjonction
$$
\E\ldrt \big ( \Sigma^d_* \Sigma^{d*} \L\big )^{\mathfrak{S}_d}.
$$
Si $U\drt \Div^d_{\Fqb}$ est un morphisme quasi-pro-étale surjectif  avec $U$ perfectoïde tel que $\E_{|U}\simeq \Qlb$ alors après restriction à $U$  cela se ramène à vérifier que
$$
|U\times_{\Div^d_{\Fqb}} (\Div^1_{\Fqb})^d |/\mathfrak{S}_d\iso |U|
$$
ce qui ne pose pas de problèmes. 
\end{proof}

\subsection{Compatibilité à la structure de monoïde sur $\Div$}

Soit $\E$ un $\Qlb$-système local de rang $1$ sur $\Div^1_{\Fqb}$ et $\E^{(d)}$ le $\Qlb$-système local de rang $1$ sur $\Div^d_{\Fqb}$ obtenu via la proposition \ref{prop: construction du symmetrise sur Qlb}.

\begin{prop}
Le système local $\coprod_{d>0} \E^{(d)}$ sur $\coprod_{d>0} \Div^d_{\Fqb}$ est  compatible à la structure de monoïde abélien sur $\coprod_{d>0} \Div^d_{\Fqb}$ au sens où via $m:\coprod_{d>0} \Div^d_{\Fqb} \times \coprod_{d>0} \Div^d_{\Fqb} \drt \coprod_{d>0} \Div^1_{\Fqb}$ on a un isomorphisme
$$
m^* \coprod_{d>0} \E^{(d)} \xrig{\ \sim\ } \coprod_{d>0} \E^{(d)} \boxtimes \coprod_{d>0} \F^{(d)}
$$
satisfaisant des conditions naturelles de commutativité et associativité.
\end{prop}
\begin{proof}
Soient $d,d'>0$ et 
$$
m_{d,d'}:\Div_{\Fqb}^d\times \Div_{\Fqb}^{d'}\ldrt \Div_{\Fqb}^{d+d'}
$$
qui s'inscrit dans un diagramme 
$$
\xymatrix{
 (\Div^1_{\Fqb})^{d} \times (\Div^1_{\Fqb})^{d'} \ar@{=}[r]\ar[d]_-{\Sigma^d\times \Sigma^{d'}} & (\Div^{1}_{\Fqb})^{d+d'} \ar[d]^-{\Sigma^{d+d'}} \\
 \Div_{\Fqb}^d\times \Div_{\Fqb}^{d'}\ar[r]^-{m_{d,d'}} & \Div_{\Fqb}^{d+d'}.
}
$$
La preuve du lemme \ref{lemme: caracterisation faisceaux descendus} s'adapte pour montrer que
$$
m_{d,d'}^* \E^{(d+d')} = \big [ (\Sigma^d\times\Sigma^{d'})_* ( \Sigma^d\times\Sigma^{d'})^*
m_{d,d'}^* \E^{(d+d')} \big ]^{\Sigma^d\times \Sigma^{d'}}.
$$
On conclut puisque 
\begin{eqnarray*}
( \Sigma^d\times\Sigma^{d'})^*
m_{d,d'}^* \E^{(d+d')} &=& \Sigma^{(d+d')*} \E^{(d+d')} \\
&=&  \E^{(1)\boxtimes (d+d')}
\end{eqnarray*}
\end{proof}

\section{\'Enoncé du théorème et application au corps de classe}

\subsection{Rappels sur le lemme de Drinfeld-Scholze (\cite{ScholzeBerkeley})}

Nous allons utiliser l'assertion d'unicité dans le lemme de Drinfeld-Scholze (qui est la partie \og facile \fg{} n'utilisant pas le théorème de classification des fibrés sur la courbe). Néanmoins il nous semble bon de réécrire  ce lemme entièrement dans le cadre des objets qui nous intéresse. 
Rappelons donc que d'après Scholze 
$$
\og 
\pi_1 \big ( ( \Div^1_{\Fqb} )^d\big ) = \pi_1 ( \Div^1_{\Fqb})^d\fg{}
$$
où ici un tel $\pi_1$ classifierait des revêtements {\it étales finis}. On a plus précisément le lemme suivant.

\begin{lemme}[Drinfeld-Scholze, cas abélien]
Soit $\La$ un anneau fini, $d>0$ et $\E$ un faisceau de $\La$-modules localement constant libre de rang $1$ sur $(\Div^1_{\Fqb})^d_{\et}$. Il existe alors $\E_1,\dots,\E_d$ de tels faisceaux localement constant sur $\Div^1_{\Fqb}$ tels que 
$$
\E \simeq \E_1\boxtimes \dots \boxtimes \E_d.
$$ 
De plus, $\E$ détermine la classe d'isomorphisme de la collection $(\E_1,\dots,\E_d)$ de façon unique.
\end{lemme}

\begin{rema}
Il faut prendre garde au fait que le lemme de Drinfeld-Scholze ne concerne que les revêtement {\it étales finis} du produit $(\Div^1_{\Fqb})^d$ et non ses revêtements quasi-pro-étales. 
\end{rema}

On en déduit le résultat suivant qui est celui qui nous intéresse.

\begin{lemme}\label{lemme: unicite par application de Drinfeld Scholze}
Soient $\E_1,\E_2$ deux $\Qlb$-systèmes locaux de rang $1$ sur $\Div^1_{\Fqb}$ de la forme 
$\mathscr{G}\otimes_{\overline{\Z}_\ell}\Qlb$ où  $\mathscr{G}$ est un $\overline{\Z}_\ell$-système local.
Si $\E_1^{(d)}\simeq \E_2^{(d)}$ comme systèmes locaux sur $\Div^d_{\Fqb}$ pour un $d>0$ On a alors $\E_1\simeq \E_2$.
\end{lemme}
\begin{proof}
Par application de $\Sigma^{d*}$ on obtient un isomorphisme 
$$
\E_1^{\boxtimes d} \simeq \E_2^{\boxtimes d}.
$$
Le résultat s'en déduit par application de l'unicité dans le lemme de Drinfeld-Scholze. 
\end{proof}

\subsection{Corps de classe géométrique}

Soit  $$\ph:W_E\ldrt \Qlb^\times$$ un caractère dont on veut montrer qu'il se factorise via 
\begin{eqnarray*}
f:W_E^{ab} & \twoheadrightarrow & E^\times 
\end{eqnarray*}
donné par l'action sur le module de Tate d'un groupe de Lubin-Tate.
Quitte à remplacer $\ph$ par $\ph\otimes \chi\circ f$ pour un $\chi$ bien choisi on peut supposer que
$$
\ph:W_E\ldrt \overline{\Z}_\ell^\times.
$$

Notons $\E^{(1)}_\ph$ le $\Qlb$-système local de rang $1$ associé sur $\Div^1_{\Fqb}$ (prop. \ref{prop: identification Qlb sys locaux et rep groupe de Weil}). D'après la proposition \ref{prop: reciprocite Artin} il s'agit de voir que 
$$
\E^{(1)}_\ph = \text{AJ}^{1*} \F_\chi^{(1)}
$$
pour un caractère $\chi:E^\times \drt \Qlb^\times$. 
Pour $d>0$ soit 
$$
\E^{(d)}_\ph = \big ( \Sigma^d_* \E_\chi^{(1)\boxtimes d} \big )^{\mathfrak{S}_d}
$$
le $\Qlb$-système local de rang $1$ associé sur $\Div^d_{\Fqb}$. Supposons maintenant démontré que 
pour $d\gg 0$ celui-ci descende le long de $\text{AJ}^d$ en un système local $\F_\chi^{(d)}$
$$
\E^{(d)}_\ph = \text{AJ}^{d*} \F_{\chi}^{(d)}
$$
où $\chi:E^\times \drt \Qlb^\times$. On a alors d'après le résultat le lemme \ref{lemme: unicite par application de Drinfeld Scholze} et la proposition \ref{prop: compatibilite tire en arriere} $\E_\ph^{(1)}=\text{AJ}^{1*}\F_\chi^{(1)}$ ce qui conclut.
\\

Ainsi, le fait que le morphisme $W_E^{ab}\twoheadrightarrow E^\times$ soit un isomorphisme (i.e. l'extension abélienne maximale de $E$ est engendrée par l'extension maximale non-ramifiée et les poins de torsion d'un groupe de Lubin-Tate) découle du théorème suivant.

\begin{theo}\label{theo: theoreme principal de simple connexite}
Pour $d>1$, resp. $d>2$ si $E|\Qp$, le diamant $\BB_{\Fqb}^{\ph=\pi^d}\setminus \{0\}$ est simplement connexe au sens où tout revêtement étale fini de celui-ci possède une section. 
\end{theo}

On va démontrer ce théorème dans les sections qui suivent.

\section{Le cas d'égales caractéristiques}
\label{sec: le cas d egales car}

On suppose dans cette section que $E=\Fq\llparent \pi\rrparent$. On va voir alors que dans ce cas là  de nombreux diamants qui interviennent dans les constructions précédentes sont en fait des espaces perfectoïdes. 

\subsection{Démonstration du théorème \ref{theo: theoreme principal de simple connexite}}

Rappelons que pour tout $S\in \Perf_{\Fq}$,  
$$
Y_S = \DD^*_S = \{0<|\pi|<1\} \subset \A^1_S
$$
une fibration tirviale en disques ouverts épointés de la variable $\pi$ sur $S$. On a donc lorsque 
$(R,R^+)$ est une $\Fq$-algèbre affinoïde perfectoïde 
$$
\O (Y_{R,R^+}) = \Big \{ \sum_{n\in \Z} x_n\pi^n\ |\ x_n \in R,\ \forall \rho\in ]0,1[\ \ \underset{|n|\drt +\infty}{\lim} \|x_n\| \rho^n=0\Big \}
$$
où $\|.\|$ est une norme d'algèbre multiplicative relativement aux puissances définissant la topologie de $R$. On en déduit une bijection pour tout $d>0$
\begin{eqnarray*}
(R^{00})^d &\iso & \BB (R,R^+)^{\ph=\pi^d} \\
(x_0,\dots,x_{d-1}) &\longmapsto & \sum_{i=0}^{d-1} \sum_{k\in \Z} \big [ x_i^{q^{-k}}\big ] \pi^{kd+i}.
\end{eqnarray*}
Soit $\G_d$ le $\O_E$-module formel $\pi$-divisible sur $\Fq$
 donné par 
\begin{eqnarray*}
 \G_d &=& (\widehat{\mathbb{G}}_a )^d\\
 \forall (x_0,\dots,x_{d-1}) \in (\widehat{\mathbb{G}}_a )^d,\ \ 
\pi. (x_0,\dots,x_{d-1}) &=& (F(x_{d-1}),x_0,\dots,x_{d-2}).
\end{eqnarray*}
On peut alors former le $E$-espace vectoriel formel revêtement universel de $\G_d$
$$
\widetilde{\G_d} = \underset{\times \pi}{\limp} \G_d 
$$
qui est isomorphe à $\spf \Big (\Fq \big \llbracket x_0^{1/p^\infty},\dots,x_{d-1}^{1/p^\infty}\big \rrbracket\Big )$. Il y a alors une identification de faisceaux pro-étales en $E$-espaces vectoriels
$$
\widetilde{\G_d} = \BB^{\ph=\pi^d}.
$$
En d'autres termes $\BB^{\ph=\pi^d}$ est représenté par le $\Fq$-espace adique parfait 
$$
t(\widetilde{\G_d})=
\spa \Big ( \Fq \big \llbracket x_0^{1/p^\infty},\dots,x_{d-1}^{1/p^\infty}\big \rrbracket,\Fq \big \llbracket x_0^{1/p^\infty},\dots,x_{d-1}^{1/p^\infty}\big \rrbracket\Big ).
$$
{\it Cet espace n'est pas perfectoïde car non-analytique.} Néanmoins 
\begin{eqnarray*}
\BB^{\ph=\pi^d} \setminus \{0\} &=& \spa \Big ( \Fq \big\llbracket x_0^{1/p^\infty},\dots,x_{d-1}^{1/p^\infty}\big \rrbracket,\Fq \big \llbracket x_0^{1/p^\infty},\dots,x_{d-1}^{1/p^\infty}\big \rrbracket\Big )\setminus V(x_0,\dots,x_{d-1}) \\
&=& \spa \Big ( \Fq \big \llbracket x_0^{1/p^\infty},\dots,x_{d-1}^{1/p^\infty}\big \rrbracket,\Fq \big \llbracket x_0^{1/p^\infty},\dots,x_{d-1}^{1/p^\infty}\big \rrbracket\Big )_a
\end{eqnarray*}
est perfectoïde\footnote{Dans une version non-perfectoïde l'auteur a rencontré pour la première fois ce type d'espaces dans \cite{StrauchBoundaries}}. Plus précisément, 
$$
\spa \Big ( \Fq \big \llbracket x_0^{1/p^\infty},\dots,x_{d-1}^{1/p^\infty}\big \rrbracket,\Fq \big \llbracket x_0^{1/p^\infty},\dots,x_{d-1}^{1/p^\infty}\big \rrbracket\Big )_a = \bigcup_{i=0}^{d-1} D(x_i)
$$
où 
$$
D(x_i)\ldrt \spa \big ( \Fq \llparent x_i^{1/p^\infty}\rrparent \big )
$$
est une boule ouverte perfectoïde sur le corps perfectoïde $\Fq \llparent x_i^{1/p^\infty} \rrparent$. Néanmoins l'espace tout entier n'est pas défini sur un corps perfectoïde. On vérifie en effet que 
$$
\O \Big (  \spa \Big ( \Fq \big \llbracket x_0^{1/p^\infty},\dots,x_{d-1}^{1/p^\infty}\big \rrbracket,\Fq \big \llbracket x_0^{1/p^\infty},\dots,x_{d-1}^{1/p^\infty}\big \rrbracket\Big )_a\Big ) = \Fq \big \llbracket x_0^{1/p^{\infty}},\dots,x_{d-1}^{1/p^\infty}\big \rrbracket 
$$
qui ne possède par d'unité topologiquement nilpotente. 
\\

{\it Venons en à la preuve du théorème principal \ref{theo: theoreme principal de simple connexite}}. On a 
$$
\BB^{\ph=\pi^d}\setminus \{0\} \sim \underset{n\geq 0}{\limp} \spa \Big ( \Fq \big \llbracket x_0^{1/p^n},\dots,x_{d-1}^{1/p^n}\big \rrbracket, \Fq \big \llbracket x_0^{1/p^n},\dots,x_{d-1}^{1/p^n}\big \rrbracket \Big ).
$$
 On en déduit (théorème d'approximation d'Elkik  ou pureté de Scholze couplé au fait que nos espaces sont quasicompacts quasi-séparés) que l'on a une équivalence de catégories
\begin{eqnarray*} &&
2-\underset{n\geq  0}{\limi} \text{\'Etales finis } \big /\ \spa \Big ( \Fqb \big \llbracket x_0^{1/p^n},\dots,x_{d-1}^{1/p^n}\big \rrbracket, \Fqb \big \llbracket x_0^{1/p^n},\dots,x_{d-1}^{1/p^n}\big \rrbracket \Big )_a \\
&\iso &
\text{\'Etale fini }\big / \ \BB^{\ph=\pi^d}_{\Fqb}\setminus \{0\}.
\end{eqnarray*}

Il suffit donc de montrer que tout revêtement étale fini de 
$$
\spa \big ( \Fqb \llbracket x_0,\dots,x_{d-1}\rrbracket , \Fqb \llbracket x_0,\dots,x_{d-1}\rrbracket  \big )_a
$$
possède une section. On utilise pour cela le résultat bien connu suivant de type GAGA.

\begin{prop} 
Soit $A$ un anneau $I$-adique noethérien. Il y a alors une équivalence de catégories 
$$
\text{\'Etales finis }  /\ \spec (A)\setminus V(I)\iso 
\text{\'Etales finis } /\ \spa (A,A)\setminus V(I).
$$
\end{prop}

Il suffit de montrer un tel théorème du type GAGA pour les faisceaux cohérents et les fibrés vectoriels. Pour cela une première méthode consiste à utiliser le point de vue de Raynaud (couplé à son résultat de platification par éclatements dans le cas des faisceaux cohérents, \cite{Ray2} sec. 5.4, qui est plus simple que le cas général). En effet, du point de vue de Raynaud 
$$
\og \spa (A,A)\setminus V(I) = \underset{\widetilde{\X}\drt \X}{\limp} \widetilde{\X} \fg{}
$$
où $\widetilde{\X}\drt \X$ parcourt les complétés $I$-adiques des éclatements $I$-admissibles de $\spec (A)$, $\X=\spf (A)$. D'après Raynaud tout fibré vectoriel  sur $\spa (A,A)_a$ s'étend en un fibré vectoriel sur un tel éclatement formel admissible
$\widetilde{\X}\drt \X$. On peut alors utiliser le théorème GAGF de \cite{EGAIII1} (sec. 5.7) pour algébriser un tel fibré vectoriel. On renvoie à \cite{AbbesElements} (en particulier la section 5.7)  pour plus de détails sur ce point de vue (et bien plus). 
\\

Afin de conclure la preuve de \ref{theo: theoreme principal de simple connexite} on utilise le {\it théorème de pureté de Zariski-Nagata (\cite{NagataPurity}, \cite{SGA2} exp.X théo. 5.4)} qui implique que pour $d>1$ on a une équivalence de catégories
\begin{eqnarray*} && \text{\'Etales finis } \big / \ \spec \big (\Fqb\llbracket x_0,\dots,x_{d-1}\rrbracket\big ) \\
&\iso &
\text{\'Etales finis } \big / \ \spec \big  (\Fqb\llbracket x_0,\dots,x_{d-1}\rrbracket \big )\setminus V(x_0,\dots,x_{d-1}). \\  
\end{eqnarray*}

On conclut en appliquant le lemme de Hensel qui nous dit que $\spec \big  (\Fqb\llbracket x_0,\dots,x_{d-1}\rrbracket \big )$ est simplement connexe.

\begin{rema}
Soit $C|\Fq$ un corps perfectoïde algébriquement clos. Alors 
$$
(\BB^{\ph=\pi^d}\setminus \{0\})_C = \mathring{\BBB}^{d,1/p^\infty}_C\setminus \{(0,\dots,0 ) \}
$$ 
une boule ouverte perfectoïde épointée en l'origine sur $C$. Cet espace n'est pas simplement connexe. Cela montre la nécessité de travailler dans un cadre absolu en particulier dans la section \ref{sec: le cas d inegales} qui suit.
\end{rema}

\subsection{Interprétation joaillère des fonctions symétriques élémentaires}

On n'utilise pas le lemme suivant dans ce texte. Néanmoins il nous parait intéressant de le noter.

%\begin{prop}
%Le diamant $\Div^d$ est un espace perfectoïde absolu i.e. pour tout $S\in \Perf_{\Fq}$, $\Div^d_S$ est un espace perfectoïde.
%\end{prop}

\begin{lemme} \label{lemme: iso remarquable egales}
Il y a un isomorphisme de faisceaux pro-étales
$$
\spa (E)^d/\mathfrak{S}_d \iso \spa (\O_E)^{d-1} \times \spa (E).
$$
\end{lemme}
\begin{proof}
Le diamant $\spa (E)^d$ représente le foncteur
$$
(R,R^+)\longmapsto (R^{00}\cap R^\times )^d
$$
sur les $\Fq$-algèbres affinoïdes perfectoïdes. Le faisceau $\spa (\O_E)^{d-1}\times \spa (E)$ est quant à lui donné par
$$
(R,R^+) \longmapsto (R^{00})^{d-1}\times (R^{00}\cap R^\times).
$$
Il y a alors une application fonctorielle $\mathfrak{S}_d$-invariante
\begin{eqnarray*}
\spa (E)^d &\ldrt & \spa (\O_E)^{d-1} \times \spa (E) \\
(x_1,\dots,x_d) & \longmapsto & ( \s_1(x_1,\dots,x_d),\dots,\s_d (x_1,\dots,x_d))
\end{eqnarray*}
où $(\s_i)_i$ sont les fonctions symétriques élémentaires. Pour $S\in \Perf_{\Fq}$ cela correspond à un morphisme de $S$-espaces perfectoïdes
$$
f:(\DD^*_S )^d \ldrt \DD_S^d \times \DD^*_S
$$
où $\DD$ désigne un disque ouvert et $\DD^*$ un disque ouvert épointé. On vérifie que si $S=\spa (C,C^+)$ avec $(C,C^+)$ un corps affinoïde perfectoïde algébriquement clos alors l'application déduite
$$
\DD^* (C)^d \ldrt \DD^d (C)\times \DD^*(C)
$$
est surjective de fibres les $\mathfrak{S}_d$-orbites. On vérifie de plus que $f$ est quasicompact. Il en résulte que pour tout $S$, $f$ est quasi-profini et est un épimorphisme pro-étale. De la même façon le morphisme 
$$
(\DD^*_S)\times \mathfrak{S}_d \ldrt (\DD^*_S)^d\times_{\DD_S^d\times \DD^*_S} (\DD^*_S)^d
$$
est quasi-pro-étale surjectif. On en déduit le résultat.
\end{proof}

On en déduit que $\Div^d$ est un quotient pro-étale de $(\spa (\O_E)^\diamond)^{d-1}\times \spa (E)^\diamond$. On ignore si ce résultat pourrait être utile.

%\begin{lemme}
%Le morphisme 
%$$
%\spa (\O_E)^{d-1} \times \spa (E) \ldrt \Div^d
%$$
%déduit du lemme \ref{lemme: iso remarquable egales} est un isomorphisme local sur le gros site analytique perfectoïde.
%\end{lemme}
%\begin{proof}
%Il suffit de vérifier que pour tout $S$ affinoïde perfectoïde 
%$$
%\DD_S^{d-1}\times \DD^*_S \ldrt \Div^d_S
%$$
%est un isomorphisme local localement sur $\DD^*_S$. Par définition $\s_1(x_1,\dots,x_d)=x_1+\dots,x_d$ et $\s_d (x_1,\dots,x_d)=x_1\dots x_d$. Par des considérations de rayons dans des disques ouverts épointés le résultat découle alors du fait suivant. Soit $\rho\in ]0,1[$. Alors
%$$
%\inf \Big \{ \big |\rho - \rho_1^{q^{k_1}} \dots \rho_d^{q^{k_d}} \big |  \ \Big | \ (\rho_1,\dots,\rho_d)\in ]0,1[^d, \ \prod_{i=1}^d \rho_i=\rho, \ (k_1,\dots,k_d)\in \Z^d\setminus \{(0,\dots,0)\}\Big \} >0.
%$$ 
%\end{proof}

\section{Le cas d'inégales caractéristiques}
\label{sec: le cas d inegales}

\subsection{Diamants absolus et diamants: une technique de descente}
\label{sec: diamants absolus}

On prend la définition suivante.

\begin{defi}
Un espace perfectoïde, resp. diamant, absolu est un faisceau pro-étale $X$ sur $\spa (\Fq)$ tel que pour $S\in \Perf_{\Fq}$ le tiré en arrière $X_S$ soit représentable par un espace perfectoïde, resp. un diamant.
\end{defi}

Voici quelques exemples:
\begin{enumerate}
\item 
Bien sûr $\spa (\Fq)$, l'objet final non-représentable du topos pro-étale, est un espace perfectoïde absolu. 
\item 
Plus généralement soit $\X$ un schéma formel noethérien formellement de type fini sur $\spf (\Fq)$ i.e. $\X_{red}$ est de type fini. Alors $\X$ est un espace perfectoïde absolu qui n'est pas un espace perfectoïde. On a par exemple si $\X=\spf ( \Fq \llbracket x_1,\dots,x_d\rrbracket )$, 
$$
\X_S = \mathring{\BBB}^{d,1/p^\infty}_S
$$
la boule ouverte perfectoïde sur $S$ de dimension $d$. Plus généralement si $\X= \spf ( \Fq \llbracket x_1,\dots,x_d\rrbracket /I)$, $\X_S=V(I)\subset \mathring{\BB}^{d,1/p^\infty}_S$ comme sous-espace perfectoïde Zariski fermé.
\item Le faisceau pro-étale 
$$
(R,R^+)\longmapsto R^+
$$
est un espace perfectoïde absolu représenté au dessus de $S$ par la boule fermée perfectoïde de dimension $1$ sur $S$. Il n'est pas associé à un schéma formel sur $\Fq$.
\item  Pour $d>0$, $\BB^{\ph=\pi^d}$ est un diamant absolu et même un espace perfectoïde absolu associé à un schéma formel lorsque $E=\Fq\llparent \pi\rrparent$. Lorsque $E|\Qp$ et $d>1$ ce diamant absolu n'est pas un espace perfectoïde absolu.
\end{enumerate}

Notons 
$$
F=\Fq\llparent T^{1/p^\infty}\rrparent.
$$
Le lemme suivant est la remarque clef qui permet d'analyser les objets absolus à partir des objets usuels.
Par \og topologie analytique\fg{} on entend la topologie  usuelle de l'espace adique.

\begin{lemme}
Le morphisme $\spa (F)\drt \spa (\Fq)$ est un épimorphisme sur le gros site analytique perfectoïde. 
\end{lemme}
\begin{proof}
C'est une simple conséquence du fait que par définition toute algèbre perfectoïde possède une pseudo-uniformisante.
\end{proof}

On en déduit aussitôt la proposition suivante.

\begin{prop}\label{prop: descente des espaces absolus}
La catégorie des espaces perfectoïdes, resp. diamants, absolus est équivalente à la catégorie des espaces perfectoïdes, resp. diamants, sur $\spa (F)$ munis d'une donnée de descente relativement au diagramme
$$
\xymatrix@C=5mm{
\spa (F)\times \spa (F)\times \spa (F) \ar@<.8ex>[r] \ar[r]\ar@<-.8ex>[r] & \spa (F)\times \spa (F) \ar@<.8ex>[r]\ar@<-.8ex>[r] & \spa (F).
}
$$
\end{prop}

On remarquera que ce diagramme de descente est particulièrement agréable puisque {\it chacune des
flêches de ce diagramme est une fibration triviale en disque ouvert perfectoïdes épointés}. Cela provient du fait que le morphisme $\spa (F)\drt \spa (\Fq)$ \og est lisse\fg{}. Une autre remarque est que {\it l'on peut remplacer $\spa (F)$ dans la proposition précédente par $\spa (F)/\ph^\Z$}. 
Donnons un exemple d'application de ceci (nous n'utiliserons pas ce résultat dans la suite). Si $X$ est un diamant absolu par définition un revêtement étale fini de $X$ est un morphisme de faisceaux pro-étales $U\drt X$ tel que pour tout $S$ perfectoïde $U_S\drt X_S$ soit un morphisme étale fini surjectif de diamants.

\begin{lemme}
L'espace perfectoïde absolu $\spa (\Fqb )$ est simplement connexe au sens où tout revêtement étale fini possède une section.
\end{lemme}
\begin{proof}
On utilise le lemme de Drinfeld-Scholze (\cite{ScholzeBerkeley}) en égales caractéristiques. D'après celui-ci si $F=\Fqb \llparent T\rrparent$ et $d>0$
$$
\og \pi_1 \big  ( (\spa (F)/\ph^\Z)^d\big )= \pi_1 \big ( \spa (F)/\ph^\Z \big )^d \fg{}.
$$
Soit $X\drt \spa (\Fqb)$  un revêtement étale fini d'espaces perfectoïdes absolus et $A$ le $\Gamma:=\Gal (\overline{F}|F)$-ensemble fini discret associé à $X_F$. Alors $A$ définit un objet cartésien sur le diagramme de topos
$$
\xymatrix@C=5mm{
B(\Gamma\times \Gamma\times \Gamma)\ar@<.8ex>[r] \ar[r]\ar@<-.8ex>[r] & B(\Gamma\times \Gamma) \ar@<.8ex>[r]\ar@<-.8ex>[r] & B\Gamma
}
$$
(topos classifiant des $G$-ensembles). 
On en déduit que l'action de $\Gamma$ sur $A$ est triviale.
\end{proof} 

Voici une autre application du diagramme de descente précédent. En utilisant le théorème de changement de base lisse (\cite{Hu1} théo. 4.5.1) et le fait que les flêches du diagramme de descente de \ref{prop: descente des espaces absolus} sont acycliques on obtient le résultat suivant.

\begin{prop}
Soit $\La$ un anneau fini dans lequel $p$ est inversible et $X$ un $\Fqb$-schéma séparé de type fini. Notons $X_{abs}$ pour l'espace perfectoïde absolu associé à $X$ vu comme faisceau sur le gros site étale perfectoïde. On a alors
$$
H^\bullet_{\et} ( X,\La) \iso H^{\bullet}_{\et} ( X_{abs},\La).
$$
\end{prop}

\subsection{Réduction à un théorème de pureté}

En analysant la preuve du théorème \ref{theo: theoreme principal de simple connexite} en égales caractéristiques on se rend compte que le point crucial est l'utilisation du théorème de pureté de Zariski-Nagata. La preuve en inégales caractéristiques de \ref{theo: theoreme principal de simple connexite} repose sur cette observation.

\begin{lemme}\label{lemme:morphisme produit diamants}
Le morphisme de diamants \og produit de $d$ éléments\fg{} 
$$
\big (\BB^{\ph=\pi}_{\Fqb} \setminus \{0\} \big )^d\ldrt \BB^{\ph=\pi^d}_{\Fqb}\setminus \{0\}
$$
est quasi-pro-étale surjectif et induit un isomorphisme de faisceaux pro-étales
$$
(\BB^{\ph=\pi}_{\Fqb} \setminus \{0\} \big )^d\ / \underline{\Delta} \rtimes \mathfrak{S}_d \iso \BB^{\ph=\pi^d}_{\Fqb}\setminus \{0\}
$$
où $\Delta =\{(\l_1,\dots,\l_d)\in (E^\times )^d\ |\ \prod_{i=1}^d \l_i=1\}$ qui agit par multiplication.
\end{lemme}
\begin{proof}
La preuve est identique à celle de la proposition \ref{prop: Div d comme quotient par le groupe symetrique}.
\end{proof}

Le résultat suivant ramène la preuve de \ref{theo: theoreme principal de simple connexite} à un résultat de pureté. 

\begin{lemme}\label{lemme: si s etend alors trivial}
Soit $U\drt \BB^{\ph=\pi^d}_{\Fqb}\setminus \{0\}$ étale fini. S'il s'étend en un morphisme étale fini au dessus du diamant absolu $\BB^{\ph=\pi^d}_{\Fqb}$ il est alors trivial.
\end{lemme}
\begin{proof}
On utilise le morphisme 
$$
f: \big (\BB^{\ph=\pi}_{\Fqb} \setminus \{0\} \big )^d\ldrt \BB^{\ph=\pi^d}_{\Fqb}\setminus \{0\}
$$
introduit dans le lemme \ref{lemme:morphisme produit diamants} précédent. Si $U$ s'étend à $\BB^{\ph=\pi^d}_{\Fqb}$ alors $f^{-1}(U)$ s'étend en un revêtement étale fini de 
$$
\big ( \BB_{\Fqb}^{\ph=\pi} \big )^d \setminus \{ (0,\dots,0)\} \simeq \spa \Big ( \Fqb \big \llbracket x_0^{1/p^\infty},\dots,x_{d-1}^{1/p^\infty} \big \rrbracket ,\Fqb \big \llbracket x_0^{1/p^\infty},\dots,x_{d-1}^{1/p^\infty} \big \rrbracket\Big )_a
$$
dont on sait qu'il est simplement connexe lorsque $d>1$ (cf. section \ref{sec: le cas d egales car}). On a donc 
$$
f^{-1}(U) \simeq \big (\BB^{\ph=\pi}_{\Fqb} \setminus \{0\} \big )^d\ \times A
$$
où $A$ est un ensemble fini. 
Puisque l'espace perfectoïde $\big (\BB^{\ph=\pi}_{\Fqb} \setminus \{0\} \big )^d$ est  connexe 
l'action de $\Delta\rtimes \mathfrak{S}_d$ sur $f^{-1} (U)$ est donnée par une action de
ce groupe sur l'ensemble fini $A$. Remarquons maintenant la chose suivante (ce point là de la démonstration est inspiré de \cite{SGA1} chap. IV Rem. 5.8). 
Notons $G=\Delta \rtimes \Sigma_d$. 
Soit $C|\Fqb$ un corps perfectoïde algébriquement clos. Pour $x\in (\BB (C)^{\ph=\pi}\setminus \{ 0\})^d$ notons $G_x$ le stabilisateur de $x$. On remarque alors que 
$$
\big  <  G_x \big >_{x\in   (\BB (C)^{\ph=\pi}\setminus \{ 0\})^d} = G
$$
(sous-groupe engendré). De plus pour un tel point géométrique $x$
$$
A= f^{-1}(U)_x=
U_{f(x)} = A/G_x.
$$
et donc $G_x$ agit trivialement sur $A$. Le groupe $G$ agit donc trivialement sur $A$. On conclut puisque 
$$
U = f^{-1}(U)/\underline{G} = \BB_{\Fqb}^{\ph=\pi^d}\setminus \{0\}\times A
$$
comme faisceau pro-étale quotient.
\end{proof}

\subsection{Retour au monde réel: dévissage du résultat de pureté à un énoncé de pureté en géométrie rigide}

On veut appliquer le lemme \ref{lemme: si s etend alors trivial} afin de conclure quant à la démonstration de \ref{theo: theoreme principal de simple connexite}. Le problème bien sûr est que $\BB^{\ph=\pi^d}_{\Fqb}$ n'est pas un diamant mais seulement un diamant absolu. On va appliquer la proposition \ref{prop: descente des espaces absolus}. Pour cela la première étape consiste à choisir une présentation \og sympathique\fg{} de l'espace de Banach-Colmez $\BB_F^{\ph=\pi^d}$.
\\

On note $F=\Fqb \llparent T^{1/p^\infty}\rrparent$. 
Choisissons $t_1,\dots,t_d\in \BB (F)^{\ph=\pi}$ tels que pour tout $i\neq j$, $V(t_i)\cap V(t_j)=\emptyset$ comme sections de $\O(1)$ sur la courbe $X_F$. Notons pour $i\in \{1,\dots,d\}$ 
$$
\hat{t}_i = \prod_{j\neq i} t_j.
$$ 
\begin{lemme}\label{lemme: presentation de notre BC}
Il y a une suite exacte de faisceaux pro-étales sur $\spa (F)$ 
$$
0\ldrt \underline{V} \xrig{\ u\ } \big (  \BB_F^{\ph=\pi}\big )^d\xrig{\ v \ } \BB_F^{\ph=\pi^d}\ldrt 0
$$
où 
\begin{eqnarray*}
v( x_1,\dots,x_d) &=&  \sum_{i=1}^d x_i \hat{t}_i \\
V &=& \big  \{(\l_1,\dots,\l_d\in E^d\ \big  | \ \sum_{i=1}^d \l_i=0\big \} \\
u (\l_1,\dots,x_d) &=& (\l_1 t_1,\dots, \l_d t_d).
\end{eqnarray*}
\end{lemme}
\begin{proof}
On travaille sur la courbe $X_F$ (le mieux étant de travailler sur la version algébrique de \cite{Courbe} pour ce genre de problèmes). On vérifie que l'on a une suite exacte de fibrés vectoriels
$$
0\ldrt V\otimes_E \O_{X_F}\ldrt \O_{X_F} (1)^d\ldrt \O_{X_F}(d)\ldrt 0
$$
où les morphismes sont analogues à ceux donnés dans l'énoncé (l'exactitude au milieu se fait en comptant les degrés une fois prouvé que la suite est exacte à ses extrémités). On peut alors appliquer le foncteur $R\tau_*$ où $\tau: \widetilde{(X_F)}_{\text{pro-ét}}\drt \widetilde{\spa (F)}_{\text{pro-ét}}$ (cf. \cite{ArthurBanachColmez}). 
\end{proof}

Le diamant $(\BB^{\ph=\pi}_F)^d$ est un disque perfectoïde ouvert de dimension $d$ sur $F$ et $(\BB^{\ph=\pi}_F)^d\setminus \underline{V}$ en est un ouvert.
On a alors le résultat de pureté suivant.

\begin{prop}\label{prop: extension de B moins Vbar}
Si $d>2$
 il y a une équivalence
$$
\text{\'Etales finis } \big / \  (\BB_F^{\ph=\pi} )^d \iso \text{\'Etales finis } \big / \  (\BB_F^{\ph=\pi} )^d \setminus \underline{V}.
$$
\end{prop}
\begin{proof}
Fixons un groupe de Lubin-Tate $\G$ sur $\O_E$.
Soit $K|E$ un corps perfectoïde contenant les points de torsion de $\G$ et muni d'une identification $K^\flat =F$. On a alors 
$$
\BB_F^{\ph=\pi} = \widetilde{\G}_K^\flat
$$
où 
\begin{eqnarray} \label{eq: limp lut}
\widetilde{\G}=\underset{\times \pi}{\limp} \G.
\end{eqnarray}
D'après le théorème de pureté de Scholze il s'agit de montrer que 
$$
\text{\'Etales finis } \big / \ \widetilde{\G}_K^d \iso \text{\'Etales finis } \big / \ \widetilde{\G}_K^d \setminus \underline{V}
$$
où $V\subset \widetilde{\G}^d (K)$. Notons 
$$
f_n: \widetilde{\G}_K^d\ldrt \G_K^d
$$
la projection sur l'étage de niveau $n$ dans la puissance $d$-ième de la limite projective (\ref{eq: limp lut}). Ce morphisme $f_n$ est un $\pi^n\underline{T_\pi (\G)^d}$-torseur pro-étale et 
$$
f_n^{-1} ( f_n (V)) = V + \pi^n T_\pi(\G)^d
$$
avec $f_n (V)\subset \G^d (\O_K)$ profini. Puisque 
$$
\widetilde{\G}_K^d \setminus \underline{V} = \bigcup_{n\geq 0} \widetilde{\G}_K^d \setminus \underline{V + \pi^n T_\pi (\G)^d}
$$
il suffit de montrer que pour tout $n$
$$
\text{\'Etales finis } \big / \ \widetilde{\G}_K^d \iso \text{\'Etales finis } \big / \ \widetilde{\G}_K^d \setminus \underline{V+ \pi^n T_\pi(\G)^d}.
$$
Cela se déduit du théorème \ref{theo: purete perfectoide par purete classique}.
\end{proof}

\begin{coro}
Si $d>2$ il y a une équivalence 
$$
\text{\'Etales finis } \big / \ \BB_F^{\ph=\pi^d} \iso \text{\'Etales finis } \big / \ \BB_F^{\ph=\pi^d} \setminus \{0\}.
$$
\end{coro}
\begin{proof}
En effet, puisque les morphismes étales finis satisfont la descente pro-étale, 
 la catégorie des morphismes étales finis vers $\BB_F^{\ph=\pi^d}$, resp. $ \BB_F^{\ph=\pi^d} \setminus \{0\}$, est équivalente à celle des morphismes étales finis $\underline{V}$-équivariants vers 
 $(\BB_F^{\ph=\pi})^d$, resp. $(\BB_F^{\ph=\pi})^d \setminus \underline{V}$.
\end{proof}

Afin de conclure pour appliquer la proposition \ref{prop: descente des espaces absolus} il suffit de prouver le résultat suivant.

\begin{lemme}
Pour $S\in \{\spa (F)\times \spa (F), \spa (F)\times \spa (F)\times \spa (F)\}$ et $d>2$
le foncteur 
$$
\text{\'Etales finis } \big / \ \BB_S^{\ph=\pi^d} \ldrt \text{\'Etales finis } \big / \ \BB_S^{\ph=\pi^d} \setminus \{0\}
$$
est pleinement fidèle.
\end{lemme}
\begin{proof}
On traite le cas $S=\spa (F)\times \spa (F)$, l'autre cas étant identique.  Utilisant  la présentation du lemme \ref{lemme: presentation de notre BC} on se ramène à vérifier que 
$$
\text{\'Etales finis } \big / \ (\BB_S^{\ph=\pi})^d \ldrt  
\text{\'Etales finis } \big / \ (\BB_S^{\ph=\pi})^d \setminus \big ( \underline{V}_F\times_{\spa (\Fqb)}\spa (F) \big )
$$
est pleinement fidèle. Fixons une identification 
$$
(\BB^{\ph=\pi} )^d= \spf ( \Fq\llbracket X_1^{1/p^\infty},\dots,X_d^{1/p^\infty}\rrbracket)
$$
et donc
$$
(\BB^{\ph=\pi}_F )^d= \mathring{\mathbb{B}}^{d,1/p^\infty}_F
$$
une boule ouverte perfectoïde sur $F$. On a alors 
$$
(\BB^{\ph=\pi}_S )^d =
(\BB^{\ph=\pi}_F )^d\times_{\spa (\Fqb)} \spa (F) = \mathring{\mathbb{B}}^{d,1/p^\infty}_F \times_{\spa (F)} \mathring{\BBB}^{1,1/p^\infty}_F.
$$
On a de plus
$$
V\subset \mathring{\BBB}^d_F (F)
$$
comme sous-ensemble localement profini. Fixons des boules fermées $\mathcal{B}_1\subset  \mathring{\BBB}^d_F$ et $\mathcal{B}_2\subset \mathring{\BBB}_F^1$
et soit $V'=V\cap \mathcal{B}_1$. 
 Il suffit alors de montrer que
 $$
 \text{\'Etales finis } \big / \mathcal{B}_1^{1/p^\infty}\times \mathcal{B}_2^{1/p^\infty} \ldrt
  \text{\'Etales finis } \big / (\mathcal{B}_1^{1/p^\infty} \setminus \underline{V'} )\times \mathcal{B}_2^{1/p^\infty} 
 $$
 est pleinement fidèle. D'après le théorème d'approximation d'Elkik on peut enlever les perfectisés i.e. il suffit de montrer que le foncteur
 $$
 \text{\'Etales finis } \big / \mathcal{B}_1\times \mathcal{B}_2\ldrt
  \text{\'Etales finis } \big / (\mathcal{B}_1 \setminus \underline{V'} )\times \mathcal{B}_2 
 $$
 est pleinement fidèle. Cela résulte du même résultat pour les fibrés vectoriels. En effet, d'après Lütkebohmert (\cite{LutkeVektor}) tout fibré vectoriel sur $\mathcal{B}_1\times \mathcal{B}_2$ est trivial. Le résultat se déduit alors d'un résultat de type Hartogs: pour $\rho \in |F|\cap ]0,1[$
 $$
 \O ( \BBB^d \times \BBB^1 ) = \O \big (  \BBB^d\setminus \mathring{\BBB}^d (\rho) \times \BBB^1)
 $$
(cf. le début de la preuve de \ref{theo: purete en geo rigide}). 
\end{proof}

\subsection{Un théorème de pureté en géométrie rigide}
\label{sec: purete en geo rigide}

\begin{theo}\label{theo: purete en geo rigide}
 Soit $X$ un $K$-espace rigide lisse quasicompact quasi-séparé purement de dimension $d>2$ et $Z\subset X(K)$ un sous-ensemble profini. Il existe alors une famille  d'ouverts admissibles quasicompacts $(U_\a)_\a$, $U_\a\subset X\setminus Z$, qui forme un recouvrement admissible de $X\setminus Z$ et 
 tels que pour tout $i$ le foncteur de restriction à $U_\a$ induise une équivalence 
$$
\text{\'Etales finis }\big / \ X \iso \text{\'Etales finis } \big / U_\a.
$$
En particulier,
$$
\text{\'Etales finis }\big / \ X \iso \text{\'Etales finis } \big / X\setminus Z.
$$
\end{theo}

Par lissité de $X$ 
tout point de $X(K)$ possède un voisinage isomorphe à $\BBB^d$ la boule fermée de dimension $d$ sur $K$. On peut donc supposer que $X=\BBB^d$. Par compacité de $Z$ pour tout $\rho \in ]0,1[\cap |K|$ on peut trouver un nombre finie d'élément $x_1,\dots,x_n\in Z$  tels que 
$$
Z\subset \coprod_{i=1}^n \BBB^d (x_i,\rho ) \subset \BBB^d.
$$
 On peut alors trouver $\rho',\rho''\in ]0,1[\cap |K|^{1/\infty}$ tels que 
$$
\rho < \rho' <\rho'' < |x_i-x_j|\text{ si } i\neq j.
$$
Considérons l'ouvert
$$
U= \BBB^d\setminus \coprod_{i=1}^n  \mathring{\BBB} (x_i,\rho')
$$
 où $\mathring{\BBB}$ désigne  une boule ouverte.
Lorsque $\rho$, $\rho'$ et $\rho''$ varient de tels ouverts forment un recouvrement admissible de $X\setminus Z$.
 Il suffit maintenant de montrer que pour tout 
$i$
$$
\text{\'Etales finis }\big /\ 
\BBB (x_i,\rho'') \iso \text{\'Etales finis } \big / \ \BBB (x_i,\rho'' ) \setminus \mathring{\BBB} (x_i,\rho').
$$
C'est exactement ce que dit la proposition suivante (la descente relativement à l'extension des scalaires vis à vis d'une extension de degré fini de $K$ implique  qu'on peut supposer que nos rayons sont dans $|K|$).

\begin{prop}
Pour $\rho\in ]0,1[\cap |K|$  et $d>2$ on a une équivalence de catégories
$$
\text{\'Etales finis } \big / \ \BBB^d \iso \text{\'Etales finis } \big / \  \BBB^d \setminus \mathring{\BBB}^d (\rho).
$$
\end{prop}
\begin{proof}
La pleine fidélité résulte du lemme d'Hartogs \ref{lemme: Hartog} qui suit. D'après le théorème \ref{theo: extension faisceaux coherents} de Lütkebohmert tout revêtement étale fini de $\BBB^d\setminus \mathring{\BBB}^d$ s'étend en un revêtement fini de $\BBB^d$ que l'on peut supposer normal (i.e. on dispose d'un \og main théorème de Zariski\fg{} dans ce cadre là, ce qui n'est pas vrai à priori en général en géométrie rigide). On conclut en utilisant la proposition  \ref{prop: etale fini sur bord implique etale fini partout}.
\end{proof}

\begin{lemme}\label{lemme: Hartog} 
Pour $d>1$ le  foncteur de restriction
$$
\text{Fibrés vectoriels sur } \BBB^d\ldrt  \text{Fibrés vectoriels sur } \BBB^d\setminus \mathring{\BBB} (\rho)
$$
est pleinement fidèle. 
\end{lemme}
\begin{proof}
D'après \cite{LutkeVektor} tout fibré vectoriel sur $\BBB^d$  est trivial. Il suffit alors de vérifier que 
$$
\O ( \BBB^d) \iso \O \big ( \BBB^d\setminus \mathring{\BBB} (\rho) \big )
$$
ce qui est immédiat par un calcul.
\end{proof}

\begin{theo}[Lütkebohmert \cite{LutkeHab}] \label{theo: extension faisceaux coherents}
Pour $d>2$ et $\rho \in ]0,1[\cap |K|$ tout fibré vectoriel $\E$ sur $\BBB^d\setminus \mathring{\BBB} (\rho)$ s'étend canoniquement en un faisceau cohérent sur $\BBB^d$ au sens où 
$
\GG \big ( \BBB^d\setminus \mathring{\BBB} (\rho), \E \big ) 
$
est un $\O(\BBB^d)$-module de type fini tel que 
$$
\GG\big  ( \BBB^d\setminus \mathring{\BBB} (\rho), \E \big ) \otimes_{\O(\BBB^d)} \O_{\BBB^d\setminus \mathring{\BBB}(\rho)} \iso \E.
$$
\end{theo}

\begin{rema}
Le théorème précédent de Lütkebohmert est plus général que le cas des fibrés vectoriels, il est vrai en profondeur $\geq 3$. Il s'agit de l'analogue en géométrie rigide de résultats de Siu en géométrie analytique complexe (\cite{SiuThesis}) (cf. également \cite{Frisch}, en particulier pour un exemple de faisceau cohérent sur un disque épointé de dimension $2$ ne s'étendant pas en un faisceau cohérent). L'analogue en géométrie algébrique de ces résultats est dû à Michèle Raynaud (\cite{SGA2}).
\end{rema}

\begin{prop}\label{prop: etale fini sur bord implique etale fini partout}
Soit $U\drt \BBB^d$ avec $d>1$ un morphisme fini avec $U$ normal et $\partial \BBB^d=\BBB^d\setminus \mathring{\BBB}^d$.
  Si $U_{|\partial \BBB^d}$ est étale au dessus de $\partial \BBB^d$ alors $U\drt \BBB^d$ est étale.
\end{prop}
\begin{proof}
Soit $\mathfrak{p}=(f)$ un idéal premier de hauteur $1$ dans l'anneau factoriel $A=\O(\BBB^d)$.  D'après le lemme  \ref{lemme: intersection non vide avec le bord} $V(f)\cap \partial \BBB^d\neq \emptyset$. On peut donc trouver des points $x_1,\dots, x_r\in (\partial \BBB^d)^{an}$, l'espace de Berkovich associé à $\partial\BB^d$, tels que si $\mathscr{I} = \O_{\BBB^d} f$, 
$$
\forall g\in \O(\BBB^d),\ g(x_1)=\dots=g(x_r)=0 \ssi g_{|\partial \BBB^d} \in \GG ( \partial \BBB^d, \mathscr{I})
$$
(il suffit de prendre la réunion de tous les points du bord de Shilov de $V(f)\cap U$ lorsque $U$ parcourt un recouvrement affinoïde admissible fini de $\partial \BBB^d$). Puisque $\times f: \O_{\BBB^d}\xrig{\sim} \mathscr{I}$, d'après Hartogs 
$$
\GG ( \partial \BBB^d,\mathscr{I} ) = \O(\BBB^d)f.
$$
Soient alors $\mathfrak{q}_1,\dots,\mathfrak{q}_r\in \spec (A)$ les supports des valuations associées à $x_1,\dots,x_r$. On a donc
$$
\mathfrak{p} = \bigcap_{i=1}^r \mathfrak{q}_i
$$
et donc $\mathfrak{p}=\mathfrak{q}_{i}$ pour un indice $i$. Il existe donc $x\in \partial (\BBB^d)^{an}$  tel que 
$$
A_{\mathfrak{p}}\ldrt \O_{(\BBB^{d})^{an},x}.
$$
Si $B=\O(U)$ le morphisme $\spec (B)\drt \spec (A)$ est étale en $\mathfrak{p}$ puisque c'est le cas après tiré en arrière à $\spec (\O_{(\BBB^{d})^{an},x})$. On conclut en appliquant le théorème de pureté de Zariski-Nagata puisque $B$ est normal.
\end{proof}

\begin{lemme}\label{lemme: intersection non vide avec le bord}
Soit $d>1$ et $f\in \O(\BBB^d)$ qui n'est pas une unité alors $V(f)\cap \partial \BBB^d\neq\emptyset$.
\end{lemme}
\begin{proof}
Par quasi-compacité de $V(f)$, si $V(f)\cap \partial \BBB^d=\emptyset$ il existe  $\rho\in ]0,1[\cap |K|$ tel que 
$$
V(f)  \subset \BBB^d (\rho).
$$
On en déduit que 
$$
\GG ( \BBB^d, \O_{\BBB^d}/ f) \iso \GG ( \BBB^d (\rho), \O_{\BBB^d}/ f).
$$
Puisque cet opérateur de restriction est complètement continu cela implique que $\O (V(f))$ est de dimension finie sur $K$. C'est impossible puisque $\dim V(f)>0$.
\end{proof}

\begin{rema}
Un autre démonstration plus élémentaire est la suivante. On peut supposer que $\| f\|_\infty=1$.
Soit alors $\tilde{f}\in k_K [X_1,\dots,X_d]\setminus \{0\}$ sa réduction. Puisque $\tilde{f}$ n'est pas une unité et $d>1$ il existe $i\in \{1,\dots,d\}$ tel que $\tilde{f}\notin k_K[X_1,\dots,X_d][X_i^{-1}]^\times$. On conclut grâce à cela. Cependant la démonstration moins élémentaire précédente s'applique dans des contextes plus généraux en rapport avec le bord de Shilov des affinoïdes.
\end{rema}

Voici au final le résultat que nous utilisons dans ce texte.

\begin{theo}\label{theo: purete perfectoide par purete classique}
Soit $X$ un $K$-espace rigide lisse quasicompact quasi-séparé purement de dimension $d>2$ et $Z\subset X(K)$ un sous-ensemble profini. Supposons nous donné un système projectif d'espaces rigides $(X_n)_{n\geq 0}$, $X_0=X$, à morphismes de transition étales finis et un espace perfectoïde $X_\infty$ tel que 
$$
X_\infty \sim \underset{n\geq 0}{\limp} {X_n}.
$$ 
On  a alors 
$$
\text{\'Etales finis }\big / \ X_\infty \iso \text{\'Etales finis } \big / X_\infty \setminus f^{-1}(Z)
$$
où $f:X_\infty \drt X$. 
\end{theo}
\begin{proof}
Il suffit de prendre les ouverts $U_\a$ du théorème \ref{theo: purete en geo rigide}. En effet, 
\begin{eqnarray*}
2-\underset{n\geq 0}{\limi} \text{\'Etales finis }\big / \ X_n & \iso & \text{\'Etales finis } \big /\ X_\infty \\
 2-\underset{n\geq 0}{\limi} \text{\'Etales finis }\big / X_n\times_X U_\a &\iso & \text{\'Etales finis } \big /\ X_\infty\times_X U_\a
\end{eqnarray*}
puisque $X$ et $U_\a$ sont quasicompacts quasi-séparés. Remarquons maintenant que si $Y\drt X_n\times_X U_\a$ est étale fini et $\overline{Y}\drt X$ désigne  l'extension canonique du morphisme étale fini composé
$$
Y\ldrt X_n\times_X U\ldrt U_\a
$$
alors par pleine fidélité de la restriction au dessus de $U$ le morphisme $Y\drt X_n\times_X U_\a $
s'étend de façon unique en un morphisme $\overline{Y}\drt X_n$.
\end{proof}

\begin{rema}
Ofer Gabber a expliqué à l'auteur pouvoir démontrer le théorème précédent lorsque $d=2$ et le corps $K$ est de caractéristique $0$. La méthode de démonstration de Gabber repose sur les résultats de \cite{Lutke1} et est différente de la précédente lorsque $d>2$.  
 Cela permettrait  donc améliorer le théorème \ref{theo: theoreme principal de simple connexite} et de montrer que $\BB^{\ph=\pi^2}_{\Fqb}\setminus \{0\}$ est simplement connexe.
\end{rema}

\bibliographystyle{plain}
\bibliography{biblio}
\end{document}